\documentclass[sn-mathphys]{sn-jnl}

 \usepackage{moreverb}

\usepackage{amsmath}

\usepackage{latexsym}

\usepackage{graphicx}
\usepackage{color}
\usepackage{amsfonts}
\usepackage{amscd}
\usepackage{mathrsfs}
\usepackage{bm}
\usepackage{enumerate}

\numberwithin{equation}{section}

    \newfont{\bg}{cmr10 scaled\magstep5}
    
    \newcommand{\bigzerou}{\smash{\lower1.7ex\hbox{\bg 0}}}

     \newcommand{\M}{\mathcal{M}}
       
      \newcommand{\PP}{\mathcal{P}}

    \newcommand{\s}{\textbf{s}}
     \newcommand{\n}{\textbf{n}}

      \newcommand{\q}{\textbf{q}}

        \newcommand{\QQ}{\mathcal{Q}}
           \newcommand{\LL}{\mathcal{L}}

\jyear{2021}%

\theoremstyle{thmstyleone}%
\newtheorem{theorem}{Theorem}
\newtheorem{proposition}[theorem]{Proposition}%

\theoremstyle{thmstyletwo}%

\theoremstyle{thmstylethree}%
\newtheorem{lemma}{Lemma}%

\begin{document}

\title[Nonlocal Manifold Poisson Models]{A Second-Order Nonlocal Approximation for Manifold Poisson Model with Dirichlet Boundary  \thanks{This 
        work of YZ and ZS  were supported by NSFC grant 12071224, 11671005.}}


\author[1]{\fnm{Yajie} \sur{Zhang}}\email{z0005149@zuel.edu.cn}

\author*[2]{\fnm{Zuoqiang} \sur{Shi}}\email{zqshi@tsinghua.edu.cn}
\equalcont{These authors contributed equally to this work.}

\affil[1]{\orgdiv{Department of Statistics and Mathematics}, \orgname{Zhongnan University of Economics and Laws}, \orgaddress{\street{Minzu Avenue}, \city{Wuhan}, \postcode{430082}, \state{Hubei}, \country{China}}}

\affil*[2]{\orgdiv{Department of Mathematical Sciences}, \orgname{Tsinghua University}, \orgaddress{\street{Zhongguancun Avenue}, \city{Beijing}, \postcode{100084}, \state{Beijing}, \country{China}}}


\abstract{Recently, we constructed a class of nonlocal Poisson model on manifold under Dirichlet boundary with global $\mathcal{O}(\delta^2)$ truncation error to its local counterpart, where $\delta$ denotes the nonlocal horizon parameter. In this paper, the well-posedness of such manifold model is studied. We utilize Poincare inequality to control the lower order terms along the $2\delta$-boundary layer in the weak formulation of model.
The second order localization rate of model is attained by combining the well-posedness argument and the truncation error analysis. Such rate is currently optimal among all nonlocal models. Besides, we implement the point integral method(PIM) to our nonlocal model through 4 specific numerical examples to illustrate the quadratic rate of convergence on the other side.}




\maketitle

\section{Introduction}

Partial differential equations on manifolds have been applied in many areas including material science \cite{CFP97} \cite{EE08} , fluid flow \cite{GT09} \cite{JL04}, biology physics \cite{BEM11}  \cite{ES10} \cite{NMWI11}, machine learning   \cite{belkin2003led} \cite{Coifman05geometricdiffusions}  \cite{LZ17}  \cite{MCL16}  \cite{reuter06dna} and image processing   \cite{CLL15}    \cite{Gu04}  \cite{KLO17}  \cite{LWYGL14}   \cite{LDMM}   \cite{Peyre09}    \cite{Lui11}. Among all the manifold PDEs in the literature, the Poisson model have been studied most frequently as it is mathematically interesting and usually reveals much information of the manifold. One recent approach in the numerical analysis of Poisson model is its nonlocal approximation. The advantage for nonlocal model is that it always avoids the use of spatial differential operator, hence new meshless numerical scheme can be explored.
Due to the difficulty for mesh generation on manifolds and the demand of solving manifold Poisson model numerically,  it is necessary to propose a certain nonlocal manifold Poisson model that can accurately approximate its  local counterpart, while being able to be solved by proper meshless numerical scheme on the other hand.

In this paper, we mainly analyze a particular nonlocal model that accurately approximates the following Poisson model:
\begin{equation}  \label{b01}
\begin{cases}
-\Delta_{\M} u(\textbf{x})=f(\textbf{x}) & \textbf{x} \in \M; \\
u(\textbf{x})=0 & \textbf{x} \in \partial \M.
\end{cases}
\end{equation}
Here $\M$ is a compact, smooth $m$ dimensional manifold embedded in $\mathbb{R}^d$, with $\partial \M$ a smooth $(m-1)$ dimensional curve with bounded curvature. $f$ is an $H^2$ function on $\M$.  
$\Delta_{\M}$ is the Laplace-Beltrami operator on $\M$. See \cite{YajieTrunc} in the 2nd page for the definition of $\Delta_{\M}$. It is well known that the boundary value problem \eqref{b01} has a unique solution $u \in H^4(\M)$.

Before we start to introduce our model, let us first review the existing nonlocal Poisson models in the literature.
In fact, most of the nonlocal Poisson models were analyzed in Euclid domains, among which the most commonly studied equation is
\begin{equation} \label{intro1}
\frac{1}{\delta^2} \int_{\Omega} (u_{\delta}(\textbf{x})-u_{\delta}(\textbf{y}) ) R_\delta(\textbf{x},\textbf{y}) d \textbf{y}=f(\textbf{x}), \qquad \textbf{x} \in \Omega.
\end{equation}
Here $\Omega \subset \mathbb{R}^k$ is a bounded Euclid domain with smooth boundary,  $f \in H^2(\Omega)$, $\delta$ is the nonlocal horizon parameter that describes the range of nonlocal interaction, $R_{\delta}(\textbf{x}, \textbf{y}) =C_{\delta} R \big( \frac{\vert \textbf{x}-\textbf{y} \vert^2}{4 \delta^2}  \big)$ is the nonlocal kernel function,
where $R \in C^2(\mathbb{R}^+) \cap L^1 [0, \infty)$ is a properly-chosen positive function with compact support, and  $C_{\delta}=\frac{1}{(4\pi \delta^2)^{k/2}}$ is the normalization factor. Such equation usually appeared in the discussion of peridynamics models \cite{Yunzhe4} \cite{Yunzhe8} \cite{Yunzhe13} \cite{Yunzhe27} \cite{Yunzhe31} \cite{Yunzhe32}.
For various kind of boundary conditions, efforts have been made to approximate $\Delta u=f$ with \eqref{intro1} by adding proper terms into \eqref{intro1} along the boundary layer, 
see \cite{Yunzhe5} \cite{Yunzhe6} \cite{Yunzhe12} \cite{Yunzhe14} for Neumann boundary condition and  \cite{Yunzhe2} \cite{book-nonlocal} \cite{Du-SIAM} \cite{Yunzhe25}  \cite{ZD10} for other types of boundary conditions. Those modifications yield to $\mathcal{O}(\delta)$ convergence rate from $u_{\delta}$ to $u$. 

 As a breakthrough,  in one dimensional \cite{Yunzhe} and two dimensional  \cite{Neumann_2nd_order} cases,  the nonlocal models with $\mathcal{O} (\delta^2)$ convergence rate to its local counterpart were successfully constructed under Neumann boundary condition. One year later, Lee H. and Du Q. in \cite{Leehwi} introduced a nonlocal model under Dirichlet boundary condition by imposing a special volumetric constraint along the boundary layer, which assures $\mathcal{O} (\delta^2)$ convergence rate in 1d segment and 2d plain disk.
  
In 2018, nonlocal Poisson model was first extended into manifold in \cite{Base1} under homogeneous Neumann boundary, where the following nonlocal Poisson model was constructed:
\begin{equation}  \label{base}
 \int_\M \frac{1}{\delta^2} R_{\delta}(\textbf{x},\textbf{y}) (u_{\delta}(\textbf{x})-u_{\delta}(\textbf{y})) d \mu_{\textbf{y}}=\int_{\mathcal{M}} f(\textbf{y}) \bar{R}_{\delta}(\textbf{x},\textbf{y}) d \mu_{\textbf{y}},
\end{equation}
 here $\M$ is the $m$ dimensional manifold embedded in $\mathbb{R}^d$ and $f \in H^2(\M)$, $d \mu_\textbf{y}$ is the volume form of $\M$. The kernel functions $R_{\delta}$ and $\bar{R}_{\delta}$ are defined as follows
 \begin{equation} \label{kernels12}
 R_{\delta}(\textbf{x}, \textbf{y}) =C_{\delta} R \big( \frac{\vert \textbf{x}-\textbf{y}\vert ^2}{4 \delta^2} \big),  \bar{R}_{\delta}(\textbf{x}, \textbf{y}) =C_{\delta} \bar{R} \big( \frac{\vert \textbf{x}-\textbf{y}\vert ^2}{4 \delta^2} \big), 
 \end{equation}
 where $\bar{R}(r)=\int_r^{+\infty} R(s) ds$ and $C_{\delta}=\frac{1}{(4\pi\delta^2)^{m/2}}$ is the centralization constant.
 The kernel function $R(r)$ is assumed to have the following constraints:
\begin{enumerate}
\item Smoothness: $\frac{d^2}{dr^2}R(r)$ is bounded, i.e., for any $r \geq 0$ we have $\big\vert \frac{d^2}{dr^2}R(r) \big\vert \leq C$;
\item Nonnegativity: $R(r)>0$ for any $r \geq 0$;
\item Compact support: $R(r)=0$ for any $r>1$;
\item Nondegenearcy: $\exists \ \delta_0 >0$ so that $R(r) \geq \delta_0 >0$ for $0 \leq r \leq 1/2$.
\end{enumerate}
In fact, the compact support assumption can be relaxed to exponentially decay, like Gaussian kernel. 

 The error of \eqref{base} has been rawly analyzed in \cite{Base1}.
 Such model has convergence rate $\mathcal{O}(\delta)$ to its local counterpart. See other nonlocal manifold models \cite{Base2} \cite{WangTangJun} with Dirichlet boundary, and \cite{Yjcms1} with interface. 
 The $\mathcal{O}(\delta)$ convergence rate was reached in \cite{LSS}, where the Dirichlet boundary was approximated by Robin condition.

In this work, to further raise the accuracy of approximation to \eqref{b01}, we propose the following nonlocal Poisson model:
\begin{equation} \label{c01}
\begin{cases}
\mathcal{L}_{\delta} u_{\delta}(\textbf{x}) - \mathcal{G}_{\delta} v_{\delta}(\textbf{x})
=  \mathcal{P}_{\delta} f(\textbf{x}) , &  \textbf{x} \in \M, \\
\mathcal{D}_{\delta} u_{\delta}(\textbf{x}) + \tilde{R}_{\delta}(\textbf{x}) v_{\delta}(\textbf{x}) =  \mathcal{Q}_{\delta} f(\textbf{x}),  & \textbf{x} \in \partial \M.
\end{cases}
\end{equation}

where the operators are defined as 
\begin{equation}
\mathcal{L}_{\delta} u_{\delta}(\textbf{x})=\frac{1}{ \delta^2} \int_{\M} (u_{\delta}(\textbf{x})-u_{\delta}(\textbf{y})) \ {R}_{\delta} (\textbf{x}, \textbf{y}) d \mu_\textbf{y} ,
\end{equation}
\begin{equation}
 \mathcal{G}_{\delta} v_{\delta}(\textbf{x})= \int_{\partial \M} v_{\delta} (\textbf{y}) \ (2+  (\textbf{x}-\textbf{y}) \cdot \kappa_{\n}  (\textbf{y}) \n(\textbf{y}) ) \ \bar{R}_{\delta} (\textbf{x}, \textbf{y})  d \tau_\textbf{y} ,
 \end{equation}
 \begin{equation}
\mathcal{D}_{\delta} u_{\delta}(\textbf{x}) = \int_{\M} u_{\delta} (\textbf{y}) \ (2 - \ (\textbf{x}-\textbf{y}) \cdot \kappa_{\n} (\textbf{x}) \n(\textbf{x}) )  \ \bar{R}_{\delta} (\textbf{x}, \textbf{y}) d \mu_\textbf{y},
\end{equation}
\begin{equation} \label{tilder}
\tilde{R}_{\delta}(\textbf{x}) =4 \delta^2 \int_{\partial \M}  \overset{=}{R}_{\delta} (\textbf{x}, \textbf{y}) d \tau_\textbf{y} - \int_{\M} \kappa_{\n}(\textbf{x}) \ ((\textbf{x}-\textbf{y}) \cdot n(\textbf{x}) )^2 \ \bar{R}_{\delta} (\textbf{x},\textbf{y}) d \mu_\textbf{y},
\end{equation}
\begin{equation}
\mathcal{P}_{\delta} f(\textbf{x})= \int_{\M} f(\textbf{y}) \ \bar{R}_{\delta} (\textbf{x}, \textbf{y}) d \mu_\textbf{y} -\int_{\partial \M} ((\textbf{x}-\textbf{y}) \cdot  \n(\textbf{y}) )  \ f(\textbf{y}) \ \bar{R}_{\delta} (\textbf{x}, \textbf{y}) d \tau_\textbf{y},
 \end{equation}
 \begin{equation}
 \mathcal{Q}_{\delta} f(\textbf{x})=-2\delta^2 \int_{\M} f(\textbf{y}) \ \overset{=}{R}_{\delta} (\textbf{x}, \textbf{y}) d \mu_\textbf{y}.
\end{equation}

We aim to approximate $(u, \frac{\partial u}{\partial \n})$ in \eqref{b01} by $(u_{\delta}, v_{\delta})$,
here $\n$ is the outward unit normal vector in $\partial \M$,  $\frac{\partial u}{\partial \n} =\nabla_{\M} u \cdot \n$,
$\kappa_{\n}$ is the constant defined in the lemma 3.3 of \cite{YajieTrunc}. 
$d \tau_\textbf{y}$ is the volume form of  $\partial \M$.
$R$ and $\bar{R}$ are the kernel functions in \eqref{base}, with the same constraints on $R$ assumed. The third kernel function $\overset{=}{R}_{\delta}(\textbf{x}, \textbf{y}) =C_{\delta} \overset{=}{R} \big( \frac{\vert \textbf{x}-\textbf{y}\vert ^2}{4 \delta^2} \big)$, where 
\begin{equation} \label{kernels3}
\overset{=}{R}(r)=\int_r^{+\infty} \bar{R}(s)ds.
\end{equation}
In general, the first equation of \eqref{c01} is an optimization of \eqref{base}  by adding some high order terms along the inner $2\delta-$layer of the boundary $\partial \M$. The second equation of \eqref{c01} is the volumetric version of Dirichlet boundary condition, where the operator $\mathcal{D}_{\delta}$ is constructed accordingly with $\mathcal{G}_{\delta}$. The idea of construction of \eqref{c01} and its truncation error analysis are presented in \cite{YajieTrunc}. 

Our purpose in this paper is to analyze the well-posedness of the model \eqref{c01}, its second order convergence to its local counterpart, and the numerical simulation of model by point integral method(PIM, see \cite{LSS} \cite{weightedLaplacian} ) that illustrates such convergence rate. Our analytic results can be easily generalized into the case with non-homogeneous Dirichlet boundary condition. To the author's best knowledge, even in the Euclid spaces, no work has ever appeared on the construction of nonlocal Poisson model with second order convergence under dimension $d \geq 3$, having such model will result in higher efficiency in the numerical implementation. In addition, as it is almost impossible to construct a mesh for high dimensional manifold, the PIM brings much more convenience than manifold finite element method(FEM). In comparison to the other meshfree numerical methods \cite{DongGuoZhi} \cite{GMLS1} \cite{KBM} \cite{localkernel} \cite{GBPF} \cite{DiffMap} on solving manifold PDEs, PIM can be easily applied to high dimensional manifold.

The paper is organized as follows: we first state our main results in section 2. Next, we describe the properties of the bilinear form corresponding to the nonlocal equations in section 3. In section 4, we analyze the well-posedness of model. The convergence of our model to \eqref{b01} is presented in section 5. In section 6, we simulate our model by point cloud method to realize such convergence rate. Finally, discussion and conclusion is included in section 7.

\section{Main Results}
Our goal in this work is to prove the following 2 theorems.
\begin{theorem}[Well-Posedness] \label{theorem3}
\begin{enumerate}
\item
For each fixed $\delta>0$ and $f \in H^1(\M) $, there exists a unique solution $u_{\delta} \in L^2(\M)$, $v_{\delta} \in L^2(\partial \M)$ to the nonlocal model $\eqref{c01}$, with the following estimate
\begin{equation} \label{wellposed1}
\left \lVert u_{\delta} \right \rVert^2_{L^2(\M)} + \delta \left \lVert v_{\delta} \right \rVert^2_{L^2(\partial \M)} \leq C \left \lVert f \right \rVert^2_{H^1(\M)}  .
\end{equation}
\item In addition, we have $u_{\delta} \in H^1(\M)$ as well, with
\begin{equation} \label{wellposed2}
\left \lVert u_{\delta} \right \rVert^2_{H^1(\M)} \leq C   \left \lVert f \right \rVert^2_{H^1(\M)} .
\end{equation}
Here the constant $C$ in the above inequalities are independent on $\delta$.
\end{enumerate}
\end{theorem}

\begin{theorem}[Vanishing Nonlocality]   \label{theorem2} 
Let $f \in H^2(\M)$, $u$ be the solution to the  Poisson model \eqref{b01}, and $(u_\delta, v_{\delta})$ be the solution to the nonlocal model \eqref{c01}, then we have the following estimate
\begin{equation}   \label{d06}
  \left \lVert u-u_\delta \right \rVert_{H^1(\M)}  + \delta^{1/2} \left \lVert \frac{\partial u}{\partial  \n} -v_\delta \right \rVert_{L^2(\partial 
 \M)} \leq C \delta^2 \left \lVert f \right \rVert_{H^2(\M)}, 
\end{equation}
where the constant is independent to $\delta$.
\end{theorem}

This two theorem indicates that \eqref{c01} assures a unique solution $u_{\delta}$ and has localization rate $\mathcal{O}(\delta^2)$ to \eqref{b01} in $H^1$ norm.
Such rate attains more accuracy than the model introduced in \cite{Base1} and is currently optimal among all the  high dimensional nonlocal models even in the case of Euclid domains.
 In the following section, some coercivity properties of the model \eqref{c01} will be given. The proof of theorem \ref{theorem3} and \ref{theorem2} will then be given separately in section \ref{main1} and \ref{main2}.

 \section{Bilinear Form of Model}

Let the functions $m_{\delta}, p_{\delta} \in L^2(\M)$,  $n_{\delta},q_{\delta} \in L^2(\partial \M)$ and satisfy the equations
\begin{equation} \label{c16}
\begin{cases}
\mathcal{L}_{\delta} m_{\delta}(\textbf{x}) - \mathcal{G}_{\delta} n_{\delta}(\textbf{x})
=  p_{\delta}(\textbf{x}) , &  \textbf{x} \in \M, \\
\mathcal{D}_{\delta} m_{\delta}(\textbf{x}) + \tilde{R}_{\delta}(\textbf{x}) n_{\delta}(\textbf{x}) = q_{\delta}(\textbf{x}),  & \textbf{x} \in \partial \M,
\end{cases}
\end{equation}
In this section, we aim to find some relations between the functions $(m_\delta, n_\delta)$ and $(p_\delta, q_\delta)$, to be the lemmas that helps to prove theorem \ref{theorem3} and \ref{theorem2}.

To begin with, for any $w_{\delta} \in L^2(\M)$, $s_{\delta}  \in L^2(\partial \M)$, we define the following bilinear function:
\begin{equation} \label{c02}
\begin{split}
& B_{\delta}[m_{\delta}, n_{\delta}; w_{\delta}, s_{\delta}] = \int_{\M} w_{\delta}(\textbf{x}) (\mathcal{L}_{\delta} m_{\delta}(\textbf{x}) - \mathcal{G}_{\delta} n_{\delta}(\textbf{x})) d \mu_\textbf{x}  \\
 & + \int_{\partial \M} s_{\delta}(\textbf{x}) (\mathcal{D}_{\delta} m_{\delta}(\textbf{x}) +  \tilde{R}_{\delta}(\textbf{x}) n_{\delta}(\textbf{x})) d \tau_\textbf{x}  =  \int_{\M} w_{\delta}(\textbf{x}) \mathcal{L}_{\delta} m_{\delta}(\textbf{x}) d \mu_\textbf{x} 
 \\ & 
 -\int_{\M} w_{\delta}(\textbf{x}) \mathcal{G}_{\delta} n_{\delta}(\textbf{x}) d \mu_\textbf{x}  
+\int_{\partial \M} s_{\delta} (\textbf{x})  \mathcal{D}_{\delta} m_{\delta}(\textbf{x}) d \tau_\textbf{x}+  \int_{\partial \M}   s_{\delta} (\textbf{x})  n_{\delta}(\textbf{x}) \tilde{R}_{\delta}(\textbf{x})  d \tau_\textbf{x},
\end{split}
\end{equation}
and the weak formulation of the equation $\eqref{c16}$ :
\begin{equation} \label{c18}
\begin{split}
B_{\delta}[m_{\delta}, n_{\delta}; w_{\delta}, s_{\delta}] =\int_{\M} w_{\delta}(\textbf{x}) p_\delta(\textbf{x})  d \mu_\textbf{x} + \int_{\partial \M} s_{\delta} (\textbf{x})  q_\delta(\textbf{x}) d \tau_\textbf{x}, \\
 \forall \  w_{\delta} \in L^2(\M), \ s_{\delta}  \in L^2(\partial \M).
 \end{split}
\end{equation}

Since $w_\delta$ and $s_\delta$ are arbitrary $L^2$ functions, it is clear that the weak formulation \eqref{c18} is equivalent to the nonlocal model \eqref{c16}. We then write down two auxiliary lemmas for $B_{\delta}$.

\begin{lemma}[Non-Negativity of the Bilinear Form] \label{theorem1}
we have
\begin{equation} \label{Bdelta}
B_{\delta}[m_{\delta}, n_{\delta};  m_{\delta}, n_{\delta}]= \frac{1}{2 \delta^2} \int_{\M}  \int_{\M}  (m_{\delta}(\textbf{x})-m_{\delta}(\textbf{y}))^2 \ {R}_{\delta} (\textbf{x}, \textbf{y}) d \mu_\textbf{x} d \mu_\textbf{y} +   \int_{\partial \M}  n^2_{\delta}(\textbf{x}) \tilde{R}_{\delta}(\textbf{x})  d \tau_\textbf{x}.
\end{equation}

\end{lemma}

\begin{lemma} \label{lemmam}
We define the weighted average functions of $m_{\delta}$ in $\M$:
\begin{equation}
\bar{m}_{\delta}(\textbf{x})= \frac{1}{\bar{\omega}_{\delta}(\textbf{x})} \int_{\M} m_{\delta} (\textbf{y}) \bar{R}_{\delta} (\textbf{x}, \textbf{y}) d \mu_\textbf{y}, \qquad \hat{m}_{\delta}(\textbf{x})=\frac{1}{\omega_{\delta}(\textbf{x})} \int_{\M} m_{\delta}(\textbf{y}) R_{\delta}(\textbf{x},\textbf{y})d \mu_\textbf{y}, 
\end{equation}
where $ \omega_{\delta}(\textbf{x})=\int_{\M} R_{\delta}(\textbf{x},\textbf{y})d \mu_\textbf{y} , \ \bar{\omega}_{\delta}(\textbf{x})=\int_{\M} \bar{R}_{\delta}(\textbf{x},\textbf{y})d \mu_\textbf{y} $,

then
\begin{equation} \label{c12}
\frac{1}{2 \delta^2} \int_{\M}  \int_{\M}  (m_{\delta}(\textbf{x})-m_{\delta}(\textbf{y}))^2 \ {R}_{\delta} (\textbf{x}, \textbf{y}) d \mu_\textbf{x} d \mu_\textbf{y} \geq 
C \left \lVert \nabla \hat{m}_{\delta} \right \rVert^2_{L^2(\M)},
\end{equation}
\begin{equation} \label{c51}
\frac{1}{2 \delta^2} \int_{\M}  \int_{\M}  (m_{\delta}(\textbf{x})-m_{\delta}(\textbf{y}))^2 \ {R}_{\delta} (\textbf{x}, \textbf{y}) d \mu_\textbf{x} d \mu_\textbf{y} \geq 
C \left \lVert \nabla \bar{m}_{\delta} \right \rVert^2_{L^2(\M)},
\end{equation}
where the constant $C$ is independent to $\delta$ and $m_{\delta}$.
\end{lemma}

\begin{proof}[Proof of Lemma \ref{theorem1}]

We calculate each term of the bilinear form in \eqref{c02} after substituting $(w_{\delta}, s_{\delta})$ by $(m_{\delta}, n_{\delta})$:

\begin{equation} \label{c03}
\begin{split}
 \int_{\M} m_{\delta}(\textbf{x}) \mathcal{L}_{\delta} & m_{\delta}(\textbf{x}) d \mu_\textbf{x}  = 
\frac{1}{ \delta^2} \int_{\M} m_{\delta}(\textbf{x}) \int_{\M} (m_{\delta}(\textbf{x})-m_{\delta}(\textbf{y})) \ {R}_{\delta} (\textbf{x}, \textbf{y}) d \mu_\textbf{y} d \mu_\textbf{x} \\
& =\frac{1}{ \delta^2} \int_{\M} m_{\delta}(\textbf{y}) \int_{\M} (m_{\delta}(\textbf{y})-m_{\delta}(\textbf{x})) \ {R}_{\delta} (\textbf{x}, \textbf{y}) d \mu_\textbf{x} d \mu_\textbf{y} \\
& =\frac{1}{2 \delta^2} \int_{\M}  \int_{\M} (m_{\delta}(\textbf{x})-m_{\delta}(\textbf{y})) (m_{\delta}(\textbf{x})-m_{\delta}(\textbf{y})) \ {R}_{\delta} (\textbf{x}, \textbf{y}) d \mu_\textbf{x} d \mu_\textbf{y} ,
\end{split}
\end{equation}

\begin{equation} \label{c04}
\begin{split}
 \int_{\M} m_{\delta}(\textbf{x}) \mathcal{G}_{\delta} n_{\delta}(\textbf{x}) d \mu_\textbf{x} & =
 \int_{\M} m_{\delta}(\textbf{x}) \int_{\partial \M} n_{\delta}(\textbf{y}) \ (2+ \kappa(\textbf{y}) \ (\textbf{x}-\textbf{y}) \cdot \n(\textbf{y})  ) \ \bar{R}_{\delta} (\textbf{x}, \textbf{y})  d \tau_\textbf{y}  d \mu_\textbf{x} \\
& =\int_{\M} m_{\delta}(\textbf{y}) \int_{\partial \M}  \ n_{\delta}(\textbf{x}) \ (2-\kappa(\textbf{x}) \ (\textbf{x}-\textbf{y}) \cdot \n (\textbf{x}) ) \ \bar{R}_{\delta} (\textbf{x}, \textbf{y})  d \tau_\textbf{x}  d \mu_\textbf{y} \\
& =\int_{\partial \M}   n_{\delta}(\textbf{x}) \ \int_{ \M}  m_{\delta}(\textbf{y})  \ (2-\kappa(\textbf{x}) \ (\textbf{x}-\textbf{y}) \cdot \n(\textbf{x})  ) \ \bar{R}_{\delta} (\textbf{x}, \textbf{y})  d \mu_\textbf{y}  d \tau_\textbf{x}  \\
& =\int_{\partial \M} n_{\delta} (\textbf{x})  \mathcal{D}_{\delta} m_{\delta}(\textbf{x}) d \tau_\textbf{x},
\end{split}
\end{equation}

the above equation \eqref{c03} and \eqref{c04} gives
\begin{equation}
B_{\delta}[m_{\delta}, n_{\delta};  m_{\delta}, n_{\delta}]= \frac{1}{2 \delta^2} \int_{\M}  \int_{\M}  (m_{\delta}(\textbf{x})-m_{\delta}(\textbf{y}))^2 \ {R}_{\delta} (\textbf{x}, \textbf{y}) d \mu_\textbf{x} d \mu_\textbf{y} +   \int_{\partial \M}  n^2_{\delta}(\textbf{x}) \tilde{R}_{\delta}(\textbf{x})  d \tau_\textbf{x},
\end{equation}
where the cross terms are eliminated by each other.

\end{proof}

\begin{proof}[Proof of lemma \ref{lemmam}]
The first inequality is from the theorem 7 of \cite{Base1}, where the assumption on the kernel function $R$ in page 3 is utilized. For the second inequality, we apply the lemma 3 of  \cite{Base1}:
\begin{equation}
\begin{split}
 \frac{1}{2 \delta^2} \int_{\M}  \int_{\M}  (m_{\delta}(\textbf{x})-m_{\delta}(\textbf{y}))^2 \ {R}(\frac{\vert \textbf{x}-\textbf{y} \vert^2 }{4 \delta^2}) d \mu_\textbf{x} d \mu_\textbf{y}  \\
 \geq  \frac{C}{2 \delta^2} \int_{\M}  \int_{\M}  (m_{\delta}(\textbf{x})-m_{\delta}(\textbf{y}))^2 \ {R}(\frac{\vert \textbf{x}-\textbf{y} \vert^2 }{32 \delta^2}) d \mu_\textbf{x} d \mu_\textbf{y} ,
 \end{split}
 \end{equation}
hence
\begin{equation} 
\begin{split}
 \frac{1}{2 \delta^2} \int_{\M}  \int_{\M} & (m_{\delta}(\textbf{x})-m_{\delta}(\textbf{y}))^2 \ {R}_{\delta} (\textbf{x}, \textbf{y}) d \mu_\textbf{x} d \mu_\textbf{y} \\
 = &  \frac{C_\delta}{2 \delta^2} \int_{\M}  \int_{\M}  (m_{\delta}(\textbf{x})-m_{\delta}(\textbf{y}))^2 \ {R}(\frac{\vert \textbf{x}-\textbf{y} \vert^2 }{4 \delta^2}) d \mu_\textbf{x} d \mu_\textbf{y} \\
  \geq & \frac{C \ C_\delta}{2 \delta^2} \int_{\M}  \int_{\M}  (m_{\delta}(\textbf{x})-m_{\delta}(\textbf{y}))^2 \ {R}(\frac{\vert \textbf{x}-\textbf{y} \vert^2 }{32 \delta^2}) d \mu_\textbf{x} d \mu_\textbf{y} \\
 \geq & \frac{C \ C_\delta}{2 \delta^2} \int_{\vert\textbf{y}-\textbf{x}\vert \leq \delta}  \int_{\M}  (m_{\delta}(\textbf{x})-m_{\delta}(\textbf{y}))^2 {R}(\frac{\vert \textbf{x}-\textbf{y} \vert^2 }{32 \delta^2})   d \mu_\textbf{x} d \mu_\textbf{y} \\ 
 \geq & \frac{C  \ \delta_0}{2 \delta^m} \frac{C_\delta}{\delta^2} \int_{\vert\textbf{y}-\textbf{x}\vert \leq \delta}  \int_{\M}  (m_{\delta}(\textbf{x})-m_{\delta}(\textbf{y}))^2 d \mu_\textbf{x} d \mu_\textbf{y}  \\
 \geq  & \frac{C}{2 \delta^2} \int_{\M}  \int_{\M}  (m_{\delta}(\textbf{x})-m_{\delta}(\textbf{y}))^2 \ {\bar{R}}_{\delta} (\textbf{x}, \textbf{y}) d \mu_\textbf{x} d \mu_\textbf{y} 
 \geq   C \left \lVert \nabla \bar{m}_{\delta} \right \rVert^2_{L^2(\M)},
 \end{split}
\end{equation}
where the last inequality is a direct corollary of \eqref{c12}.
\end{proof}

Next, we state the main lemma in this section.
\begin{lemma}[Regularity] \label{lemma1}
For any functions $m_{\delta}, p_{\delta} \in L^2(\M)$, and $n_{\delta},q_{\delta} \in L^2(\partial \M)$ that satisfy the system of equations \eqref{c16},
\begin{enumerate}
\item
there exists a constant $C$ independent to $\delta$ such that
\begin{equation} \label{c50}
  B_{\delta}[m_{\delta}, n_{\delta}; m_{\delta}, n_{\delta}] + \frac{1}{ \delta }  \left \lVert q_{\delta} \right \rVert^2_{L^2(\partial \M)} \geq  C( \left \lVert m_{\delta} \right \rVert_{L^2(\M)}^2 + \delta \left \lVert n_{\delta} \right \rVert_{L^2(\partial \M)}^2);
\end{equation}
\item  If in addition, $p_\delta$ satisfies the following conditions
\begin{enumerate}
\item
\begin{equation} \label{assum1}
\left \lVert \nabla p_{\delta} \right \rVert_{L^2(\M)}+ \frac{1}{\delta} \left \lVert  p_{\delta} \right \rVert_{L^2(\M)} \leq F(\delta) \left \lVert  p_0 \right \rVert_{H^\beta(\M)},
\end{equation}
\item
\begin{equation} \label{assum2}
\int_{\M}  p_{\delta}(\textbf{x}) \ f_1(\textbf{x}) d \mu_\textbf{x}  \leq G(\delta) ( \left \lVert  f_1 \right \rVert_{H^1(\M)} + \left \lVert  \bar{f}_1 \right \rVert_{H^1(\M)} + \left \lVert  \overset{=}{f}_1 \right \rVert_{H^1(\M)} ) \left \lVert  p_0 \right \rVert_{H^\beta(\M)},  
\end{equation}
for all function $f_1 \in H^1(\M)$ and some function $p_0 \in H^\beta(\M)$ and some constant $F(\delta)$, $G(\delta)$ depend on $\delta$, with the notations
\begin{equation}
\bar{f}_1(\textbf{x})= \frac{1}{\bar{\omega}_{\delta}(\textbf{x})} \int_{\M} f_1 (\textbf{y}) \bar{R}_{\delta} (\textbf{x}, \textbf{y}) d \mu_\textbf{y}, \qquad \overset{=}{f}_1(\textbf{x})=\frac{1}{\overset{=}{\omega}_{\delta}(\textbf{x})} \int_{\M} f_1(\textbf{x}) \overset{=}{R}_{\delta}(\textbf{x},\textbf{y})d \mu_\textbf{y},
\end{equation}
and $$\bar{\omega}_{\delta}(\textbf{x})=\int_{\M} \bar{R}_{\delta}(\textbf{x}, \textbf{y}) d \mu_\textbf{y}, \ \overset{=}{\omega}_{\delta}(\textbf{x})=\int_{\M} \overset{=}{R}_{\delta}(\textbf{x},\textbf{y})d \mu_\textbf{y}, \qquad \forall \ \textbf{x} \in \M, $$
\end{enumerate}
then we will have $m_{\delta} \in H^1(\M)$, with the estimate
\begin{equation} \label{c55}
\left \lVert m_{\delta} \right \rVert_{H^1(\M)}^2 +  \delta \left \lVert n_{\delta} \right \rVert_{L^2(\partial \M)}^2  \leq C \ \big( (  G^2(\delta) + \delta^4 F^2(\delta)) \left \lVert  p_0 \right \rVert^2_{H^\beta(\M)}  + \frac{1}{\delta} \left \lVert q_{\delta} \right \rVert^2_{L^2(\partial \M)} \big).
\end{equation}
\end{enumerate}
\end{lemma}

This lemma gives a complete control on the bilinear form $B_{\delta}$, and is crucial in the well-posedness and convergence analysis. The main idea of proof is to apply Poincare inequality to the interior terms of $B_{\delta}$, then control the high order terms along the $2\delta$-layer of the boundary by the help of the boundary equation. We have moved the proof of such lemma into appendix due to its extensive calculation.

\section{Well-Posedness of Nonlocal Model} \label{main1}

The main purpose of this section is to prove theorem \ref{theorem3}. We will mainly apply lemma \ref{lemma1} in the proof.

\begin{proof}[Proof of Theorem \ref{theorem3}]
\begin{enumerate}
\item Recall the second equation of our model \eqref{c01}:
\begin{equation}
\mathcal{D}_{\delta} u_{\delta}(\textbf{x}) + \tilde{R}_{\delta}(\textbf{x}) v_{\delta}(\textbf{x}) =  \mathcal{Q}_{\delta} f(\textbf{x}) , \ \textbf{x} \in \partial \M,
\end{equation}
this gives
\begin{equation} \label{vdelta}
 {v}_{\delta}(\textbf{x})= \frac{ \mathcal{Q}_{\delta} f(\textbf{x}) }{\tilde{R}_{\delta}(\textbf{x})} -  \frac{\mathcal{D}_{\delta} u_{\delta}(\textbf{x}) }{  \tilde{R}_{\delta}(\textbf{x})  }, \ \textbf{x} \in \partial \M,
\end{equation}
and we apply it to the first equation of \eqref{c01} to discover
\begin{equation} \label{fix1}
\begin{split}
\mathcal{L}_{\delta}u_{\delta} (\textbf{x})+(\mathcal{G}_{\delta} \frac{\mathcal{D}_{\delta} u_{\delta} }{  \tilde{R}_{\delta} } )(\textbf{x})=\mathcal{P}_{\delta}f(\textbf{x}) +\mathcal{G}_{\delta}  (\frac{ \mathcal{Q}_{\delta} f(\textbf{x}) }{\tilde{R}_{\delta}(\textbf{x})} ), \ \textbf{x} \in \M.
\end{split}
\end{equation}
Our purpose here is to show there exists a unique solution $u_{\delta} \in L^2(\M)$ to the equation \eqref{fix1}, and thus $v_{\delta}(\textbf{x})$ can be solved by \eqref{vdelta}. In fact, according to the Lax-Milgram theorem, to present the uniqueness of $u_{\delta}$ in \eqref{fix1} and the estimate \eqref{wellposed1} for $u_{\delta}$ and $v_{\delta}$, our task can be reduced to the following 3 inequalities:
\begin{enumerate}
\item Coercivity: $$ \int_\M u_\delta(\textbf{x}) (\mathcal{L}_{\delta}u_{\delta} (\textbf{x}) + (\mathcal{G}_{\delta} \frac{\mathcal{D}_{\delta} u_{\delta} }{  \tilde{R}_{\delta} } )(\textbf{x}) ) d \mu_\textbf{x} \geq C \left \lVert u_\delta \right \rVert^2_{L^2(\M)} ,$$
\item Boundedness: for all $w_\delta \in L^2(\M)$, $$ \int_\M w_\delta(\textbf{x}) (\mathcal{L}_{\delta}u_{\delta} (\textbf{x}) + (\mathcal{G}_{\delta} \frac{\mathcal{D}_{\delta} u_{\delta} }{  \tilde{R}_{\delta} } )(\textbf{x}) ) d \mu_\textbf{x} \leq C_\delta \left \lVert u_\delta \right \rVert_{L^2(\M)}  \left \lVert w_\delta \right \rVert_{L^2(\M)},$$ 
\item Bound for right hand side: for all $w_\delta \in L^2(\M)$, $$ \int_{\M} w_{\delta}(\textbf{x}) \mathcal{P}_{\delta} f(\textbf{x}) d \mu_\textbf{x} + \int_{\M} w_{\delta}(\textbf{x}) \ \mathcal{G}_{\delta}  (\frac{ \mathcal{Q}_{\delta} f(\textbf{x}) }{\tilde{R}_{\delta}(\textbf{x})} ) d \mu_\textbf{x} \leq C \left \lVert f \right \rVert_{H^1(\M)} \left \lVert w_{\delta} \right \rVert_{L^2(\M)};$$
\end{enumerate}
where the positive constant $C_{\delta}$ in (b) depends on $\delta$, and $C$ in (a) (c) are independent on $\delta$.  We move the proof of (b) and (c) into appendix and only present (a) in this section. We denote
\begin{equation}
 \tilde{v}_{\delta}(\textbf{x})  =  \frac{\mathcal{D}_{\delta} u_{\delta}(\textbf{x}) }{  \tilde{R}_{\delta}(\textbf{x})  }, \ \textbf{x} \in \partial \M.
\end{equation}
From the proof of lemma \ref{theorem1}, we know that $$\int_{\M} \mathcal{G}_{\delta}  \tilde{v}_{\delta}(\textbf{x})   u_{\delta}(\textbf{x}) d \mu_\textbf{x}   =   \int_{\partial \M} \tilde{v}_{\delta} (\textbf{x})  \mathcal{D}_{\delta} u_{\delta}(\textbf{x}) d \tau_\textbf{x},$$ hence
\begin{equation}
\begin{split} \label{fix3}
&\int_{\M} ( \mathcal{L}_{\delta}u_{\delta} (\textbf{x})+(\mathcal{G}_{\delta} \frac{\mathcal{D}_{\delta} u_{\delta} }{  \tilde{R}_{\delta} } )(\textbf{x})  ) u_{\delta}(\textbf{x}) d \mu_\textbf{x}   
= \int_{\M} ( \mathcal{L}_{\delta}u_{\delta} (\textbf{x}) + \mathcal{G}_{\delta}  \tilde{v}_{\delta}(\textbf{x})  ) u_{\delta}(\textbf{x}) d \mu_\textbf{x}   \\
& =  \int_{\M} ( \mathcal{L}_{\delta}u_{\delta} (\textbf{x})) u_{\delta}(\textbf{x}) d \mu_\textbf{x}   +  \int_{\partial \M} \tilde{v}_{\delta} (\textbf{x})  \mathcal{D}_{\delta} u_{\delta}(\textbf{x}) d \tau_\textbf{x} \\ &= \int_{\M} ( \mathcal{L}_{\delta}u_{\delta} (\textbf{x})) u_{\delta}(\textbf{x}) d \mu_\textbf{x}   +  \int_{\partial \M} \tilde{R}_{\delta}(\textbf{x}) \tilde{v}^2_{\delta} (\textbf{x}) d \tau_\textbf{x} =  B_{\delta}[u_{\delta}, \tilde{v}_{\delta}; u_{\delta}, \tilde{v}_{\delta}] ;
\end{split}
\end{equation}
and we apply the first part of lemma \ref{lemma1} to obtain
\begin{equation}
 B_{\delta}[u_{\delta}, \tilde{v}_{\delta}; u_{\delta}, \tilde{v}_{\delta}] 
 = B_{\delta}[u_{\delta}, \tilde{v}_{\delta}; u_{\delta}, \tilde{v}_{\delta}] + \left \lVert \mathcal{D}_{\delta} u_{\delta}- \tilde{R}_{\delta}  \tilde{v}_{\delta} \right \rVert^2_{L^2(\M)}
 \geq C \left \lVert u_{\delta} \right \rVert^2_{L^2(\M)}.
 \end{equation}
Hence we have completed the proof of (a).
\item We apply a weaker argument of lemma \ref{lemma1}(i) to the model $\eqref{c01}$: if we can show 
\begin{enumerate}
\item \label{c21}
\begin{equation}
\left \lVert \nabla_{\M} (\mathcal{P}_{\delta}f) \right \rVert_{L^2(\M)}  + \frac{1}{\delta} \left \lVert \mathcal{P}_{\delta}f \right \rVert_{L^2(\M)}  \leq \frac{C}{\delta} \left \lVert f \right \rVert_{H^1(\M)}  ,
\end{equation}
and
\item \label{c20}
\begin{equation} 
\int_{\M} \mathcal{P}_{\delta}f (\textbf{x}) \  f_1(\textbf{x}) d \mu_\textbf{x} \leq C \left \lVert f \right \rVert_{H^1(\M)} \left \lVert f_1 \right \rVert_{H^1(\M)} , \qquad \forall \ f_1 \in H^1(\M);
\end{equation}
\end{enumerate}
then the second part of  lemma \ref{lemma1} will give us
\begin{equation}
\left \lVert u_{\delta} \right \rVert_{H^1(\M)}^2 +  \delta \left \lVert v_{\delta} \right \rVert_{L^2(\partial \M)}^2  \leq C \ (   \left \lVert  f \right \rVert^2_{H^1(\M)}  + \frac{1}{\delta} \left \lVert  \mathcal{Q}_{\delta}  f \right \rVert^2_{L^2(\partial \M)} +  \delta^2 \left \lVert  f \right \rVert^2_{H^1(\M)}),
\end{equation}
consequently,
\begin{equation}
\left \lVert u_{\delta} \right \rVert_{H^1(\M)}^2 \leq C \    \left \lVert  f \right \rVert^2_{H^1(\M)} .
\end{equation}

In fact, the estimate \eqref{c20} has already been shown in \eqref{c31} in the part 1 as
\begin{equation} \label{c32}
\int_{\M}  f_1(\textbf{x}) \ \mathcal{P}_{\delta} f(\textbf{x})  \ d \mu_\textbf{x}  \leq C  \left \lVert f \right \rVert_{H^1(\M)}   \left \lVert f_1 \right \rVert_{L^2(\M)}  ,
\end{equation}

so what remains to present is \eqref{c21}. Recall
\begin{equation}
\mathcal{P}_{\delta} f(\textbf{x})= \int_{\M} f(\textbf{y}) \ \bar{R}_{\delta} (\textbf{x}, \textbf{y}) d \mu_\textbf{y} +\int_{\partial \M} ((\textbf{x}-\textbf{y}) \cdot  \n(\textbf{y}) )  \ f(\textbf{y}) \ \bar{R}_{\delta} (\textbf{x}, \textbf{y}) d \mu_\textbf{y},
\end{equation}

hence
\begin{equation} \label{c33}
\begin{split}
& \left \lVert \nabla_{\M} (\mathcal{P}_{\delta}f) \right \rVert_{L^2(\M)}  + \frac{1}{\delta} \left \lVert \mathcal{P}_{\delta}f \right \rVert_{L^2(\M)} \\
 & \leq   
  \frac{1}{\delta}\left \lVert \int_{\M} f(\textbf{y}) \ \bar{R}_{\delta} (\textbf{x}, \textbf{y}) d \mu_\textbf{y} \right \rVert_{{L}_{\textbf{x}}^2(\M)}
+ \left \lVert  \nabla_{\M}^{\textbf{x}}  \int_{\M} f(\textbf{y}) \ \bar{R}_{\delta} (\textbf{x}, \textbf{y}) d \mu_\textbf{y} \right \rVert_{{L}_{\textbf{x}}^2(\M)}\\
& + \frac{1}{\delta} \left \lVert\int_{\partial \M} ((\textbf{x}-\textbf{y}) \cdot  \n(\textbf{y}) )  \ f(\textbf{y}) \ \bar{R}_{\delta} (\textbf{x}, \textbf{y}) d \mu_\textbf{y} \right \rVert_{{L}_{\textbf{x}}^2(\M)} \\
& + \left \lVert \nabla_{\M}^{\textbf{x}} \int_{\partial \M} ((\textbf{x}-\textbf{y}) \cdot  \n(\textbf{y}) )  \ f(\textbf{y}) \ \bar{R}_{\delta} (\textbf{x}, \textbf{y}) d \mu_\textbf{y} \right \rVert_{{L}_{\textbf{x}}^2(\M)}.
 \end{split}
\end{equation}
The control for the above 4 terms are exactly the same as the control for the equations \eqref{sm1} \eqref{sm2} \eqref{sm3}. As a consequence,
\begin{equation}
\left \lVert \nabla_{\M} (\mathcal{P}_{\delta}f) \right \rVert_{L^2(\M)}  + \frac{1}{\delta} \left \lVert \mathcal{P}_{\delta}f \right \rVert_{L^2(\M)}  \leq \frac{C}{\delta} \left \lVert f \right \rVert_{L^2(\M)} + \frac{C}{\delta^{\frac{1}{2}}} \left \lVert f \right \rVert_{L^2(\partial \M)} \leq \frac{C}{\delta}  \left \lVert f \right \rVert_{H^1(\M)}.
\end{equation}
Therefore we proved \eqref{c21}. Together with \eqref{c20} which is shown in \eqref{c32}, we eventually conclude $
\left \lVert u_{\delta} \right \rVert_{H^1(\M)}^2 \leq C \    \left \lVert  f \right \rVert^2_{H^1(\M)}$.

\end{enumerate}

\end{proof}

\section{Vanishing Nonlocality} \label{main2}
Our goal in this section is to prove theorem \ref{theorem2}. So far we have established the well-posedness of our model \eqref{c01}. To compare such model with its local counterpart \eqref{b01}, what we need more is the truncation error analysis between \eqref{c01} and \eqref{b01}. Fortunately, we have proved the following lemma in our previous work.
\begin{lemma}  \label{Nonlocal_Model} [Theorem 3.1 of \cite{YajieTrunc}]
Let $u \in H^4(\M)$ solves the system $\eqref{b01}$, $v(\textbf{x})=\frac{\partial u}{\partial \n}(\textbf{x}) $ for $\textbf{x} \in \partial \M$, and \begin{equation} \label{b001}
r_{in}(\textbf{x})=\mathcal{L}_{\delta} u(\textbf{x}) - \mathcal{G}_{\delta} \frac{\partial u} {\partial {\n} } (\textbf{x})
- \mathcal{P}_{\delta} f(\textbf{x}) , \qquad  \textbf{x} \in \M, 
\end{equation}
\begin{equation} \label{b002}
r_{bd}(\textbf{x})= \mathcal{D}_{\delta} u(\textbf{x}) + \tilde{R}_{\delta}(\textbf{x}) \frac{\partial u} {\partial {\n} } (\textbf{x}) - \mathcal{Q}_{\delta} f(\textbf{x}),  \qquad \textbf{x} \in \partial \M;
\end{equation} 
then we can decompose $r_{in}$ into $r_{in}=r_{it}+r_{bl}$, where $r_{it}$ is supported in the whole domain $\M$, with the following bound
\begin{equation}  \label{ddd1}
\frac{1}{\delta} \left \lVert r_{it} \right \rVert_{L^2(\M)} + \left \lVert \nabla r_{it} \right \rVert_{L^2(\M)}  \leq C \delta  \left \lVert u \right \rVert_{H^4(\M)};
\end{equation}
and $r_{bl}$ is supported in the layer adjacent to the boundary $\partial \M$ with width $2\delta$:
\begin{equation}
supp(r_{bl}) \subset \{ \textbf{x} \ \big\vert \ \textbf{x} \in \M, \ dist(\textbf{x}, \partial \M) \leq 2 \delta \ \},
\end{equation}
 and satisfy the following two estimates
\begin{equation} \label{d07}
\frac{1}{\delta} \left \lVert r_{bl} \right \rVert_{L^2(\M)} + \left \lVert \nabla r_{bl} \right \rVert_{L^2(\M)} \leq C \delta^{\frac{1}{2}}   \left \lVert u \right \rVert_{H^4(\M)};
\end{equation}
\begin{equation} \label{symme}
\begin{split}
\int_{\M}  r_{bl}(\textbf{x}) \ f_1(\textbf{x}) d \mu_\textbf{x} \leq C \delta^2 \left \lVert u \right \rVert_{H^4(\M)} ( \left \lVert  f_1\right \rVert_{H^1(\M)} + \left \lVert  \bar{f}_1\right \rVert_{H^1(\M)} +  \left \lVert  \overset{=}{f}_1\right \rVert_{H^1(\M)}  )  ,   \\ \forall \ f_1 \in H^1(\M),
\end{split}
\end{equation}
where the notations $\nabla=\nabla_{\M}$, and
\begin{equation} \label{fff1}
\bar{f}_1(\textbf{x})= \frac{1}{\bar{\omega}_{\delta}(\textbf{x})} \int_{\M} f_1 (\textbf{y}) \bar{R}_{\delta} (\textbf{x}, \textbf{y}) d \mu_\textbf{y}, \qquad \overset{=}{f}_1(\textbf{x})=\frac{1}{\overset{=}{\omega}_{\delta}(\textbf{x})} \int_{\M} f_1(\textbf{x}) \overset{=}{R}_{\delta}(\textbf{x},\textbf{y})d \mu_\textbf{y}
\end{equation}
represents the weighted average of $f_1$ in ${B}_{2\delta}(\textbf{x})$ with respect to $\bar{R}$ and $\overset{=}{R}$, and $$\bar{\omega}_{\delta}(\textbf{x})=\int_{\M} \bar{R}_{\delta}(\textbf{x}, \textbf{y}) d \mu_\textbf{y}, \ \overset{=}{\omega}_{\delta}(\textbf{x})=\int_{\M} \overset{=}{R}_{\delta}(\textbf{x},\textbf{y})d \mu_\textbf{y}, \qquad \forall \ \textbf{x} \in \M. $$

In addition, we have the following estimate for $r_{bd}$:
\begin{equation}  \label{rbd1}
\left \lVert r_{bd} \right \rVert_{L^2(\partial \M)} \leq C \delta^{\frac{5}{2}} \left \lVert u \right \rVert_{H^4( \M)} .
\end{equation}
\end{lemma}
This lemma gives a complete control on the truncation error of \eqref{c01}. Next, we apply lemma \ref{lemma1} to derive the localization rate of model under such truncation error.

\begin{proof}[Proof of theorem \ref{theorem2}]
Let us denote the error functions: $$e_{\delta}(\textbf{x})=u(\textbf{x})-u_{\delta}(\textbf{x}), \ \textbf{x} \in \M; \qquad e^{\n}_{\delta}(\textbf{x})=\frac{\partial u}{\partial \n}(\textbf{x})-v_{\delta}(\textbf{x}), \ \textbf{x} \in \partial \M.$$
 We then subtract \eqref{c01} with the equation \eqref{b001} \eqref{b002} to discover
\begin{equation} \label{d01}
\begin{cases}
\mathcal{L}_{\delta} e_{\delta}(\textbf{x}) - \mathcal{G}_{\delta} e^{\n}_{\delta}(\textbf{x})
= r_{in}, &  \textbf{x} \in \M, \\
\mathcal{D}_{\delta} e_{\delta}(\textbf{x}) + \tilde{R}_{\delta}(\textbf{x}) e_{\delta}^{\n}(\textbf{x}) = r_{bd},  & \textbf{x} \in \partial \M,
\end{cases}
\end{equation}

According to the lemma \ref{lemma1}, if the following 3 inequalities hold:
\begin{enumerate}
\item
\begin{equation} \label{d03}
\frac{1}{\delta} \left \lVert r_{in} \right \rVert_{L^2(\M)} + \left \lVert \nabla r_{in} \right \rVert_{L^2(\M)} 
 \leq C \delta^{\frac{1}{2}}   \left \lVert u \right \rVert_{H^4(\M)},
\end{equation}

\item 
\begin{equation} \label{d04}
\begin{split}
\int_{\M}  r_{in}(\textbf{x}) \ f_1(\textbf{x}) d \mu_\textbf{x} 
 \leq C \delta^2 \left \lVert u \right \rVert_{H^4(\M)} & ( \left \lVert  f_1\right \rVert_{H^1(\M)} + \left \lVert  \bar{f}_1\right \rVert_{H^1(\M)} +  \left \lVert  \overset{=}{f}_1\right \rVert_{H^1(\M)}  ) \\ & \forall \ f_1 \in H^1(\M),
\end{split}
\end{equation}

\item 
\begin{equation}  \label{d05}
\left \lVert r_{bd} \right \rVert_{L^2(\partial \M)} \leq C \delta^{\frac{5}{2}} \left \lVert u \right \rVert_{H^4( \M)} ,
\end{equation}

\end{enumerate}

then we will have the estimate
\begin{equation}
\begin{split}
\left \lVert e_{\delta} \right \rVert_{H^1(\M)}^2 +  \delta \left \lVert e^\n_{\delta} \right \rVert_{L^2(\partial \M)}^2  & \leq \delta^4 (C \delta  \left \lVert  u \right \rVert^2_{H^4(\M)})  + \frac{1}{\delta} (C \delta^5 \left \lVert u \right \rVert^2_{H^4(\M)} ) +  C \delta^4  \left \lVert  u \right \rVert^2_{H^4(\M)} \\ & \leq C \delta^4  \left \lVert  u \right \rVert^2_{H^4(\M)}.
\end{split}
\end{equation}

What left is to show the estimate \eqref{d03} \eqref{d04} and \eqref{d05}. In fact, \eqref{d03} is a direct sum of \eqref{ddd1} and \eqref{d07},  while \eqref{d05} is exactly \eqref{rbd1}. For  \eqref{d04}, we present such estimate by combining \eqref{symme} with the following inequality
\begin{equation}  \label{d09}
\int_{\M}  r_{it}(\textbf{x}) \ f_1(\textbf{x}) d \mu_\textbf{x} \leq C \delta^2 \left \lVert u \right \rVert_{H^4(\M)} \left \lVert  f_1\right \rVert_{H^1(\M)}, \qquad \forall \ f_1 \in H^1(\M),
\end{equation}
which is derived by  $\left \lVert r_{it} \right \rVert_{L^2(\M)} \leq  \delta^2 \left \lVert u \right \rVert_{H^4(\M)}$ that mentioned in \eqref{ddd1}.
We then complete our proof.
 
\end{proof}


\section{Numerical Simulation of Nonlocal Model}
 \subsection{Point Integral Method}
In conclusion, the analysis in the previous sections indicates that our model \eqref{c01} approximates the manifold Poisson problem \eqref{b01} in the quadratic rate. So far our results are all on the continuous setting. Nevertheless, a natural thinking is to solve \eqref{c01} with proper numerical scheme, so that to obtain a numerical solution for the widely-studied manifold Poisson problem \eqref{b01}.
  As we mentioned in the beginning, a corresponding numerical method named point integral method(PIM) can be applied. The main idea is to sample the manifold and its boundary with a set of sample points, which is usually called point cloud, then to approximate the integral of a function by adding up the value of the function at each sample point multiplied by its volume weight.  The calculation of volume weight involves the use of $K$-nearest neighbors to construct local mesh around each points. The numerical solution is then obtained by solving the discretized system of linear equations. See \cite{LSS} for detailed explanation of point integral method.
  
In this work, we will apply PIM to our model \eqref{c01}. Assuming we are given the input value $\delta$, the kernel function $R(r)$, the set of points $\PP=\{ \textbf{p}_i \}_{i=1}^n $ that samples $\M$, $\QQ=\{ \q_k\}_{k=1}^m$ that samples $\partial \M$;  the volume weight $\mathcal{A}=\{A_i \}_{i=1}^n$ for each $\textbf{p}_i  \in \M$, and the hypersurface weight $\LL=\{ L_k \}_{k=1}^m$ for each $\q_k \in \partial \M$.

Then we discretize \eqref{c01} into the following linear system:
\begin{equation}  \label{numericalsoln}
\begin{cases}
\sum \limits_{j=1}^n L_{\delta}^{ij} (u_i-u_j) - \sum \limits_{k=1}^m G_{\delta}^{ik} v_k = f_{1\delta}^i & i=1,2,...n.\\
\sum \limits_{j=1}^n D_{\delta}^{lj} u_j + \tilde{R}_{\delta}^l v_l=f_{2\delta}^l & l=1,2,...,m.
\end{cases}
\end{equation}
where the discretized coefficients are given as follows
\begin{equation} \label{Ldelta}
L_\delta^{ij} =\frac{1}{\delta^2} R_\delta(\textbf{p}_i, \textbf{p}_j) A_j ,
\end{equation}
\begin{equation} \label{Gdelta}
G_{\delta}^{ik} =(2 + \kappa_{n}(\q_k)(\textbf{p}_i-\q_k) \cdot \n_k) \bar{R}_\delta (\textbf{p}_i, \q_k) L_k,
\end{equation}
\begin{equation}
f_{1\delta}^i=\sum \limits_{j=1}^n f(\textbf{p}_j) \bar{R}_\delta( \textbf{p}_i, \textbf{p}_j) A_j - \sum \limits_{k=1}^m ( \textbf{p}_i-\q_k) \cdot \n_k f(\q_k) \bar{R}_\delta( \textbf{p}_i, \q_k) L_k,
\end{equation}
\begin{equation} \label{Ddelta}
D_\delta^{lj}=(2 -  \kappa_{n} (\q_l) (\q_l-\textbf{p}_j) \cdot \n_l) \bar{R}_\delta (\q_l,\textbf{p}_j) A_j ,
\end{equation}
\begin{equation}
\tilde{R}_{\delta}^l =4 \delta^2 \sum \limits_{k=1}^m \overset{=}{R}_\delta(\q_l, \textbf{p}_k) L_k - \sum \limits_{j=1}^n  \kappa_{n} (\q_l) ( (\q_l-\textbf{p}_j) \cdot \n_l )^2 \bar{R}_\delta(\q_l,\textbf{p}_j) A_j ,
\end{equation}
\begin{equation}
f_{2\delta}^l=-2\delta^2 \sum \limits_{j=1}^n f(\textbf{p}_j) \overset{=}{R}_\delta(\q_l, \textbf{p}_j) A_j .
\end{equation}
Here $\n_k=\n(\q_k)$, $R_{\delta}, \bar{R}_{\delta}$ and $\overset{=}{R}_{\delta}$ are the kernel functions defined in section 1.  The system \eqref{numericalsoln} gives $(m+n)$ linear equations on the variables $\{u_i\}_{i=1,...,n}$, $\{v_k\}_{k=1,2,...,m}$,  in the matrix form:
\begin{equation} \label{matrix1}
\begin{pmatrix}
P_{\delta}  -L_{\delta} & -G_{\delta} \\
D_{\delta}   &  \tilde{Q}_{\delta}
\end{pmatrix}
\begin{pmatrix}
U \\
V
\end{pmatrix}
=
\begin{pmatrix}
f_{1\delta} \\
f_{2\delta}
\end{pmatrix},
\end{equation}
where $L_{\delta}, G_{\delta}, D_{\delta}$ are the matrix forms of  \eqref{Ldelta}, \eqref{Gdelta}, \eqref{Ddelta};  $P_{\delta}=diag(\sum \limits_{j=1}^n L_{\delta}^{ij} )_{i=1,2,...,n} $, $ \tilde{Q}_{\delta}=diag(\tilde{R}_{\delta}^l )_{l=1,2,...,m} $, $U=\{u_i\}_{i=1,...,n}$, $V=\{v_k\}_{k=1,2,...,m}$.

As \eqref{numericalsoln} is the discretized form of \eqref{c01}, while \eqref{c01} converges to \eqref{b01} in a rate of $\mathcal{O}(\delta^2)$ according to Theorem \ref{theorem2}, our expectation is that \eqref{numericalsoln} would give a numerical solution to the Poisson problem \eqref{b01} in certain form. 
 In fact, theoretically, we can prove the following two propositions: 
\begin{proposition}
\begin{enumerate}
\item \eqref{matrix1} assures a unique solution vector $(U,V)$. 
\item Let 
$$h_1=h_1(\PP,\mathcal{A}, \M)= \sup \limits_{f_1 \in C^1(\M)} \frac{ \Big\vert \int_{\M} f_1(\textbf{y}) d \mu_{\textbf{y}} - \sum \limits_{i=1}^n f_1(\textbf{p}_i) A_i \Big\vert }  { \vert supp(f_1)\vert \left \lVert f_1 \right \rVert_{C^1(\M)} }, $$
$$h_2=h_2(\QQ, \LL, \partial \M) = \sup \limits_{g_1 \in C^1(\partial \M)} \frac{ \Big\vert \int_{\partial \M} g_1(\textbf{y}) d \tau_{\textbf{y}} - \sum \limits_{k=1}^m g_1(\textbf{p}_i) L_k \Big\vert }  { \vert supp(g_1)\vert \left \lVert g_1 \right \rVert_{C^1(\partial \M)} }$$
be the integral accuracy indexes of the point cloud $(\PP,\mathcal{A})$ that samples $\M$, $(\QQ,\LL)$ that samples $\partial \M$,  where $\vert supp(f_1)\vert, \ \vert supp(g_1)\vert $ is the volume of the support of $f_1, g_1$.

On the other hand, based on the solution vector $(U,V)$ of \eqref{matrix1}, we construct a function $I_u$ defined in $\M$:
 \begin{equation}
 \begin{split}
& I_u(\textbf{x})=\frac{1}{\sum \limits_{j=1}^n \frac{1}{\delta^2} R_\delta(\textbf{x}, \textbf{p}_j) A_j}
\Big( \sum \limits_{j=1}^n \bigg( \frac{1}{\delta^2} R_\delta(\textbf{x}, \textbf{p}_j)  u_j+
 f(\textbf{p}_j) \bar{R}_\delta( \textbf{x}, \textbf{p}_j)  \bigg) A_j \\
& +\sum \limits_{k=1}^m \bigg( (2 + \kappa_{n}(\q_k)(\textbf{x}-\q_k) \cdot \n_k) \bar{R}_\delta (\textbf{x}, \q_k) v_k
 -  ( \textbf{x}-\q_k) \cdot \n_k f(\q_k) \bar{R}_\delta( \textbf{x}, \q_k)  \bigg) L_k \Big).
 \end{split}
 \end{equation}
 then $I_u(\textbf{p}_j)=u_j$ for each $\textbf{p}_j \in \PP$, and there exist constants $C$ and $T_0$ depend only on $\M$ such that for any $t \leq T_0$,
 \begin{equation}
 \left \lVert u-I_u \right \rVert_{H^1(\M)} \leq C (\delta+\frac{h_1(\PP,\mathcal{A},\M) +h_2(\QQ,\LL,\partial \M) }{\delta^3} )  \left \lVert f \right \rVert_{C^1(\M)} ,
 \end{equation}
 where $u$ is the solution of the Poisson problem \eqref{b01}.

\end{enumerate}
\end{proposition}
For the proof proposition 1, we denote $\tilde{A}=diag(A_j)_{j=1,2,...,n}, \tilde{L}=diag(L_k)_{k=1,2,...,m}$,  then
\begin{equation} \label{matrix2}
\begin{pmatrix}
\tilde{A} & 0 \\
0 & \tilde{L}
\end{pmatrix}
\begin{pmatrix}
P_{\delta}  -L_{\delta} & -G_{\delta} \\
D_{\delta}   &  \tilde{Q}_{\delta}
\end{pmatrix}
=
\begin{pmatrix}
\tilde{A} (P_{\delta}  -L_{\delta}) & -\tilde{A} G_{\delta}    \\
 \tilde{L}D_{\delta}  &  \tilde{L}  \tilde{Q}_{\delta}
\end{pmatrix}.
\end{equation}
By definition, $\tilde{A} (P_{\delta}  -L_{\delta}) $ is symmetric and diagonally dominated with positive diagonal; $ \tilde{L}  \tilde{Q}_{\delta}$ is a diagonal matrix with all positive elements; and $ \tilde{A} G_{\delta}= (\tilde{L}D_{\delta})^T$. Apparently, the determinant of \eqref{matrix2} is positive, which assures the uniqueness of solution of \eqref{matrix1}. \\

We omit the proof of proposition 2 since several advanced tools are utilized, see the proof of Theorem 1 of \cite{Base1} for reference. Proposition 2 correlates $(U,V)$ with the solution $u$ of \eqref{b01}. If the point cloud $(\PP,\mathcal{A})$, $(\QQ,\LL)$ are sufficiently accurate, by choosing appropriate $\delta=\delta(h_1,h_2)$, we have
\begin{equation}
\lim \limits_{h_1, h_2 \to 0}  \left \lVert u-I_u \right \rVert_{H^1(\M)}=0.
\end{equation}
This indicates that $I_u$ approximates $u$ as $h_1, h_2 \to 0$. 
In the remaining part of this section, we will do several specific numerical examples to study numerically the rate of convergence from $I_u$ to $u$.
Due to the difficulty on the control of $(u-I_u)$ through implementation, we recall $I_u(\textbf{p}_j)=u_j$ and employ two alternative functions that refer to the relative error between $u$ and $(U,V)$, where only the values on the point cloud are counted:
 \begin{equation} \label{ee2}
\mbox{relative interior} \ l^2 \ \mbox{error: } \ e_2=\sqrt{ \frac{
 \sum \limits_{j=1}^{n} ( u_{j}-u(\textbf{p}_{j}))^2 A_{j} }{ \sum \limits_{j=1}^{n}  u^2(\textbf{p}_j) A_{j} }}
 \end{equation}
 
  \begin{equation} \label{ee2b}
\mbox{relative boundary} \ l^2 \ \mbox{error: } \ e_2^b=\sqrt{ \frac{ 
 \sum \limits_{k=1}^m (v_k- {\frac{\partial u}{\partial \n}} (\q_k) )^2 L_k }
{ \sum \limits_{k=1}^m v^2(\q_k) L_k   } }
 \end{equation}

In addition, for all the numerical examples in this section, we choose the following kernel function $R$ for convenience:
\begin{equation} \label{kernelr}
R(r)=
\begin{cases}
\frac{1}{2}(1+ \cos \pi r ), & 0 \leq r \leq 1, \\
0, & r>1.
\end{cases}
\end{equation}
so that the functions $R_{\delta}, \bar{R}_{\delta}$ and $\overset{=}{R}_{\delta}$ in \eqref{numericalsoln} can be calculated through the definitions \eqref{kernels12} \eqref{kernels3} when $\delta$ is given. 

\subsection{Example: 2D Hemisphere Embedded in $\mathbb{R}^3$ } \label{hmsphere}
In the first example, we let the manifold $\M$ be the upper half of the unit hemisphere, with equation
\begin{equation}
x^2+y^2+z^2=1, \qquad z \geq \frac{1}{2},
\end{equation}
while its boundary $\partial \M$ be the unit circle $x^2+y^2=1, z =\frac{1}{2}$. 
In the local Poisson problem \eqref{b01}, we let the exact solution $u$ be $u(x,y,z)=z^3-\frac{1}{4}z$. By the definition of $\nabla_{\M}$ and $\Delta_{\M}$, we parametrize $\M$ and calculate 
\begin{equation}
\frac{\partial u}{\partial \n}(x,y,\frac{1}{2}) =-\frac{\sqrt{3}}{4}, \qquad f=-\Delta_{\M} u(x,y,z)=12z^3-\frac{13}{2}z.
\end{equation}

Next, to solve for $(U,V)$ from \eqref{numericalsoln}, we implement the following iteration through Matlab:

for $t=4:11$
\begin{enumerate} 
\item choose $\delta=1/t$, and let $n=2t^4+4t^2$, $m=4t^2$;
\item set two random vectors $P=rand(n,1), Q=rand(n-m,1)$, then for each integer $i \in [1,n-m]$, let $\textbf{p}_i=\big( \sqrt{1- (\frac{Q(i)+1}{2})^2} \cos (2\pi P(i)), \sqrt{1- (\frac{Q(i)+1}{2})^2} \sin(2\pi P(i)),  \frac{Q(i)+1}{2} \big)$; for each $i \in [n-m+1,n]$, let $\textbf{p}_i=\big( \frac{\sqrt{3}}{2} \cos (2\pi P(i)),  \frac{\sqrt{3}}{2} \sin(2\pi P(i)),  \frac{1}{2} \big)$;
\item  let $\q_k=\textbf{p}_{k+n-m}, i=1,2,...,m$;
\item for each $i \in [1,n]$, we first find $20$ points from the point cloud $\PP$ that are closest to $\textbf{p}_i$, then project them onto the tangent plane of $\M$ at $\textbf{p}_i$;
\item we rebuild the coordinate of these 21 points on such plane, with $\textbf{p}_i$ located at $(0,0)$. Based on the new coordinates, we construct a $2d$ triangulation of these 21 points using Matlab function delaunay;
\item in such triangulation, we collect all the triangles that has $\textbf{p}_i$ as its one vertex. The value of $A_i$ is then assigned to be $1/3$ the sum of areas of these triangles;
\item we sort the array $X=[ P(n-m+1), ... , P(n)]$ from small to large, to write as $Y=[X(i_1),X(i_2),...,X(i_m)]$. Then for each $i_k \in [1,m]$, we assign
$ L_{i_k}= \frac{1}{2}( \big\vert \q_{i_{k+1}}-\q_{i_{k}} \big\vert+\big\vert \q_{i_k}-\q_{i_{k-1}} \big\vert) $. Here $\q_{i_0}=\q_{i_m}$, $\q_{i_{m+1}}=\q_{i_1}$;
\item calculate $f(\q_k), f(\textbf{p}_i), R_{\delta}, \bar{R}_{\delta}$ and $\overset{=}{R}_{\delta}$. In this specific example, $\kappa_{\n}(\q_k) =-\frac{\sqrt{3}}{3}$, $\n_k = <\frac{\sqrt{3}}{3} \q_k(1), \frac{\sqrt{3}}{3} \q_k(2), -\frac{\sqrt{3}}{2}>$. We then complete the stiff matrix of \eqref{numericalsoln};
\item we use function GMRES to solve the system \eqref{numericalsoln} and obtain the solution vector $(U,V)$.
By comparing it with $u$ and $\frac{\partial u}{\partial \n}$, we output the value $e_2$ and $e_2^b$ defined in \eqref{ee2} and \eqref{ee2b}. 
\end{enumerate} 
end for \\
 
In step $2-3$, we randomly generate an $n$-point cloud on $\M$ and an $m$-point cloud on $\partial \M$ to assure that the points are basically uniformly distributed. The choice of $n$ and $m$ guarantees that the nonlocal horizon $\delta$ is approximately $0.75\sqrt{h}$, where $h$ refers to the average mesh size. This empirically brings more accuracy. Step $4-6$ provides a method to generate a volume weight vector $\mathcal{A}$ for the point cloud $\PP$, where the use of $K$-nearest points method is involved on the construction of local mesh around each point. see Algorithm $7$ in page 14 of \cite{Base1} on how to generate a volume weight vector for a point cloud of general dimensional manifold. The volume weight vector $\LL$ of $\QQ$ is provided in step 7. We record the output $e_2, \ e_{2b}$ in each iteration, and their rate of change with respect to $\delta$ between each consecutive iterations in the Table \ref{figure61}.
\begin{table}[htb] 
\begin{tabular}{|c|c|c|c|c|c|c|}
\hline
$\delta$ & n & m    & $e_2$ &  rate of $e_2$ w.r.t $\delta$ & $e_2^b$ & rate of $e_2^b$ w.r.t. $\delta$ \\
\hline
0.250 &  576 & 64 &   0.0158 &  N/A &  0.0465   &  N/A \\
\hline
0.200 & 1350 & 100 & 0.0099 &  2.0950 &  0.0288   &  2.1469  \\
\hline
0.167 & 2736 & 144  & 0.0078 &  1.3076    &  0.0221   &  1.4524  \\       
         \hline 
0.143 & 4998  & 196   & 0.0056 & 2.1496  &  0.0125   & 3.6967   \\
\hline
0.125 & 8448 &256  & 0.0040 &  2.5198  &   0.0098  &  1.8224  \\
\hline
0.111 & 13446 & 324  & 0.0033 &  1.6333  &  0.0078    & 1.9380  \\
\hline
0.100 & 20400 & 400  & 0.0026  &  2.2628 &   0.0062   & 2.1789  \\       
\hline 
0.091 & 29766  & 484    & 0.0020 & 2.7527   &   0.0051   & 2.0492  \\
\hline
\end{tabular}
\caption{ \label{figure61} error of PIM and rate of convergence between each iteration in example 1.}
\label{diagram1}
\end{table}

Next, to evaluate geometrically the rate of convergence, we plot the $8$ points $(\ln(\delta), \ln(e_2))$ from each iteration on the $2D$ rectangular coordinate system, and sketch the auxiliary line $y=2x-1.34$; besides, we plot the $8$ points $(\ln(\delta), \ln(e_{2b}))$ on the same plane and sketch the auxiliary line $y=2x-0.45$, to obtain Figure \ref{figure62}.
\begin{figure}[htb] 
\centering
\includegraphics[width=.90\textwidth]{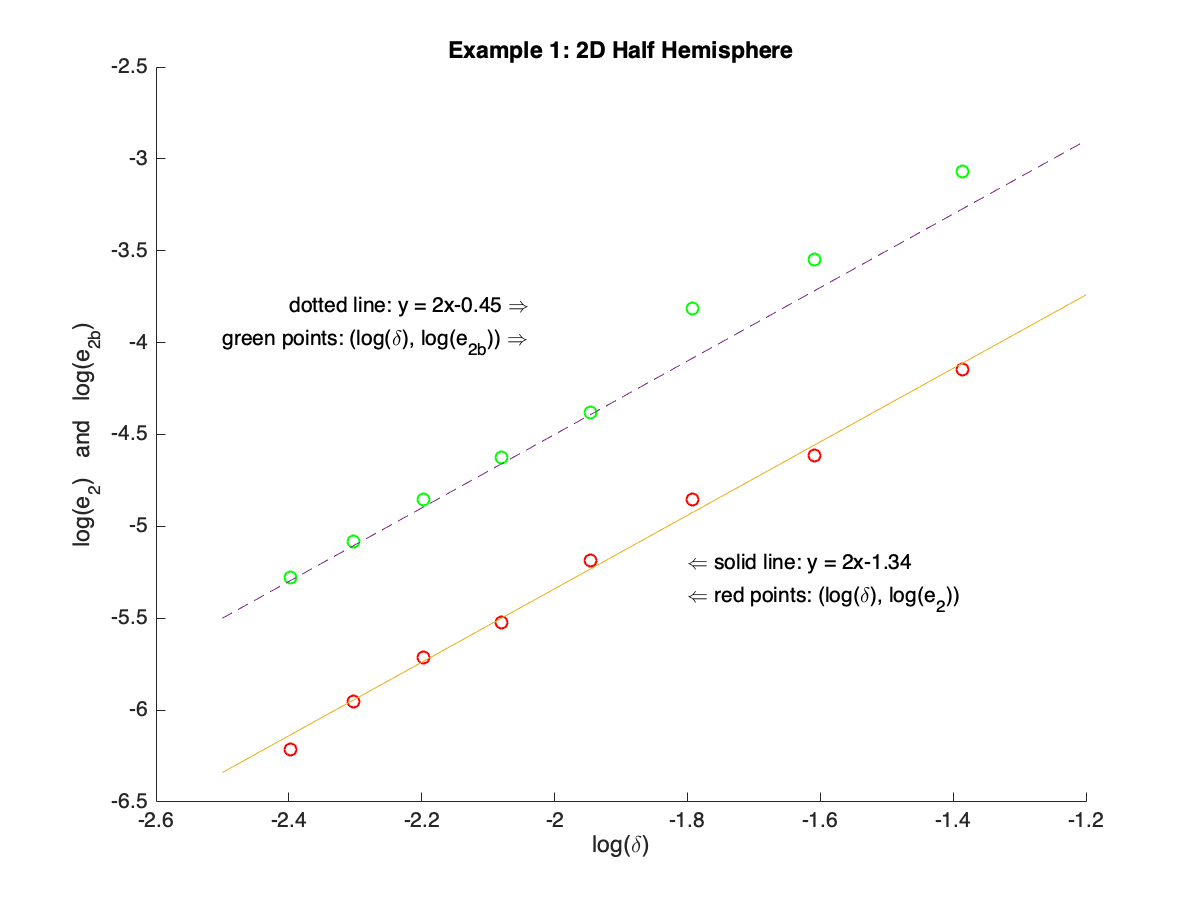}
\caption{   \label{figure62} $l^2$ approximation error vs $\delta$ and their fitting lines in example 1. }
 \end{figure}

Figure \ref{figure62} indicates that $e_2, \ e_{2b} $ are almost linearly dependent on $\delta^2$, where the perturbation mainly comes from the randomness of point cloud. Next, we introduce another example in higher dimensional manifold.

\subsection{3D Manifold Embedded in $\mathbb{R}^4$} \label{3dball}
In the second example, we let $\M$ be the following manifold embedded in $\mathbb{R}^4$:
\begin{equation}
x^2+y^2+z^2+w^2=1, w \geq 0,
\end{equation}
with its boundary $\partial \M$ be the unit ball $x^2+y^2+z^2=1, \ w=0$.
In the local Poisson problem \eqref{b01}, we let the exact solution $u$ be $u(x,y,z,w)=xw$. By the definition of $\Delta_{\M}$, we parametrize $\M$ and calculate 
\begin{equation}
f(x,y,z,w)=-\Delta_{\M} u=8xw, \qquad  \frac{\partial u}{\partial \n}(x,y,z,0)=-x.
\end{equation}

Still, we implement the following iteration through Matlab to solve for $(U,V)$ from \eqref{numericalsoln}:

for $t=2:6$
\begin{enumerate} 
\item choose $\delta=1/t$, and let $n=3t^6+4t^4$, $m=4t^4$;
\item set the random $3 \times n$ matrix $P=rand(n,3)$, then for each integer $i \in [1,n-m]$, let $\textbf{p}_i=\big( \sqrt{P(i, 1)} \cos (2\pi P(i,2)), $ $ \sqrt{P( i, 1)} \sin(2\pi P(i,2)),  \sqrt{1-P(i, 1)}\cos (\pi P(i,3)),  \sqrt{1-P(i, 1)}\sin (\pi P(i,3))   \big)$; for each $i \in [n-m+1,n]$, let $\textbf{p}_i=\big( \sqrt{1-(2P(i,2)-1)^2} \cos (2\pi P(i,1)), \sqrt{1-(2P(i,2)-1)^2}  \sin(2\pi P(i,1)), $ $  2P(i,2)-1 \big)$;
\item  let $\q_k=\textbf{p}_{k+n-m}, i=1,2,...,m$;
\item for each $i \in [1,n]$, we first find $50$ points from $\PP$ that are closest to $\textbf{p}_i$, then project them onto the $3d$ tangent hyperplane of $\M$ at $\textbf{p}_i$;
\item we rebuild the coordinate of these $51$ points on such hyperplane, or equivalently, $3d$ space, with $\textbf{p}_i$ located at $(0,0,0)$. Based on the new coordinates, we construct a $3d$ triangulation of these $51$ points using Matlab function delaunay;
\item for each $i \in [1,m]$, we first find $20$ points from $\QQ$ that are closest to $\q_k$, then project them onto the tangent plane of $\M$ at $\q_k$;
\item we rebuild the coordinate of these $21$ points on such plane, with $\q_k$ located at $(0,0)$. Based on the new coordinates, we construct a $2d$ triangulation of these 21 points using Matlab function delaunay;
\item in such triangulation, we collect all the triangles that has $\q_k$ as its one vertex. The value of $L_i$ is then assigned to be $1/3$ the sum of areas of these triangles;
\item calculate $f(\q_k), f(\textbf{p}_i), R_{\delta}, \bar{R}_{\delta}$ and $\overset{=}{R}_{\delta}$. In this specific example, $\kappa_{\n}(\q_k) \equiv 0$, $\n_k \equiv <0, 0,0,-1>$. We then complete the stiff matrix of \eqref{numericalsoln};
\item we use function GMRES to solve the system \eqref{numericalsoln} and obtain the solution vector $(U,V)$.
By comparing it with $u$, we output the value $e_2$ and $e_2^b$ defined in \eqref{ee2} and \eqref{ee2b}.
\end{enumerate} 
end for \\

In step $2-3$, we randomly generate an $n$-point cloud on $\M$ and an $m$-point cloud on $\partial \M$ to assure that the points are basically uniformly distributed. The choice of $n$ and $m$ guarantees that the nonlocal horizon $\delta$ is approximately $0.75\sqrt{h}$, where $h$ refers to the average mesh size.  The volume weight vector $\mathcal{A}$ for the point cloud $\PP$ is provided in step 4-5, while $\LL$ for $\QQ$ is provided in step 6-8, in which the K-nearest points method is applied.
Same as example 1, we record $e_2$ and $e_{2b}$ from each iteration in Table \ref{figure63} and compute their rate of change with respect to $\delta$ between each consecutive iteration. Besides, in Figure \ref{figure64}, we plot the 10 points $(\ln(\delta), \ln(e_2))$ and $(\ln(\delta), \ln(e_{2b}))$ on the 2D rectangular coordinate system and sketch 2 auxiliary lines $y=2x-0.7$, $y=2x-1.95$ to better explain the rate of convergence.

\begin{table}[htb] 
\begin{tabular}{|c|c|c|c|c|c|c|}
\hline
$\delta$ & n & m    & $e_2$ &  rate of $e_2$ w.r.t $\delta$ & $e_2^b$ & rate of $e_2^b$ w.r.t. $\delta$ \\
\hline
0.500 &  256 & 64 &   0.0477 &  N/A &  0.2248   &  N/A \\
\hline
0.333 &  2511 & 324 &   0.0156 &  2.7565 &  0.0628   &  3.1452 \\
\hline
0.250 &  13312 & 1024 &   0.0096 &  1.6877 &  0.0384  &  1.7099 \\
\hline
0.200 & 49375 & 2500 & 0.0065 &  1.7476 &  0.0195  &  3.0368  \\
\hline
0.167 & 145152 & 5184  & 0.0040 &  2.6629    &  0.0121   &  2.6174  \\       
         \hline 
\end{tabular}
\caption{ \label{figure63}  error of PIM and rate of convergence between each iteration in example 2.}
\label{diagram3}
\end{table}
\begin{figure}[htb]  
\centering
\includegraphics[width=.90\textwidth]{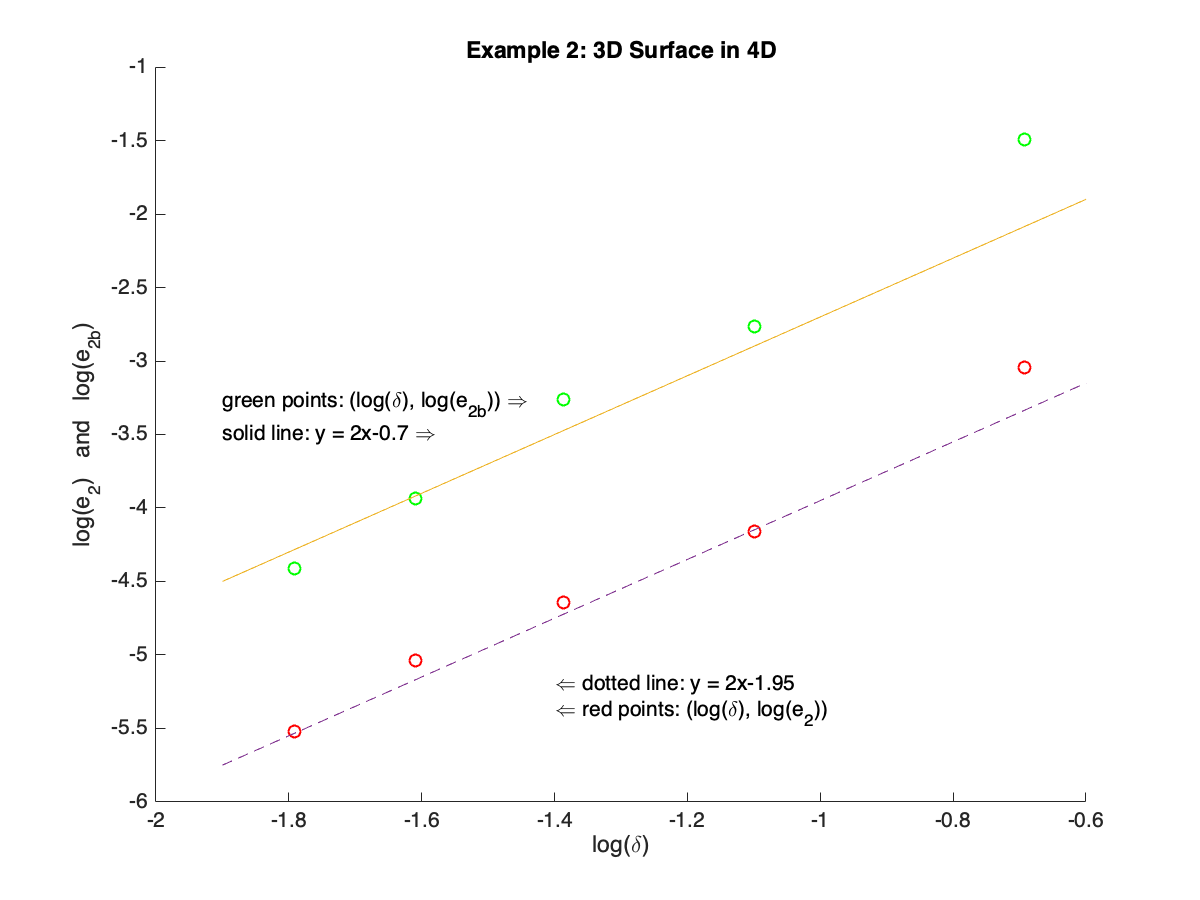}
\caption{ \label{figure64} $l^2$ approximation error vs $\delta$ and their fitting lines in example 2. }
 \end{figure}

Figure \ref{figure64} indicates that $e_2, \ e_{2b} $ are almost linearly dependent on $\delta^2$ as well as in example 1.
The results in the previous two examples imply that the discrete solution of \eqref{numericalsoln} generated by PIM converges to the exact solution of \eqref{b01} in a rate of $\mathcal{O}(\delta^2)$ in the discrete $l^2$ norm, which is $\mathcal{O}(h)$ where $h$ refers to  the average mesh size. 
Example 1 and 2 numerically illustrate the localization rate of our nonlocal model in Theorem \ref{theorem3} on the other side.

\subsection{Numerical Test for Non-homogeneous Boundary} 

We now extend our local Poisson problem into the case with non-homogeneous Dirichlet boundary:
\begin{equation}  \label{bb01}
\begin{cases}
-\Delta u(\textbf{x})=f(\textbf{x}) & \textbf{x} \in \M; \\
u(\textbf{x})=g(\textbf{x}) & \textbf{x} \in \partial \M.
\end{cases}
\end{equation}
where $g \in H^3(\partial \M)$.
By analyzing again the truncation error part in \cite{YajieTrunc}, we omit the proof to establish the following nonlocal Poisson model that approximates \eqref{bb01}:
\begin{equation} \label{numericalsoln1}
\begin{cases}
\mathcal{L}_{\delta} u_{\delta}(\textbf{x}) - \mathcal{G}_{\delta} v_{\delta}(\textbf{x})
=  \mathcal{P}_{\delta} f(\textbf{x})+\mathcal{S}_\delta g(\textbf{x}) , &  \textbf{x} \in \M, \\
\mathcal{D}_{\delta} u_{\delta}(\textbf{x}) + \tilde{R}_{\delta}(\textbf{x}) v_{\delta}(\textbf{x}) =  \mathcal{Q}_{\delta} f(\textbf{x})+ \tilde{P}_\delta (\textbf{x}) g(\textbf{x}) ,  & \textbf{x} \in \partial \M.
\end{cases}
\end{equation}
where the operator
\begin{equation}
\mathcal{S}_\delta g(\textbf{x})=-\int_{\partial \M } ((\textbf{x}-\textbf{y}) \cdot \n(\textbf{y})) \ \Delta_{\partial \M} \ g(\textbf{y}) \bar{R}_\delta(\textbf{x},\textbf{y}) d \mu_\textbf{y},
\end{equation}
and the function
\begin{equation}
\tilde{P}_\delta (\textbf{x})= \int_{\M}  (2 - \kappa_{\n} (\textbf{x}) \ (\textbf{x}-\textbf{y}) \cdot \n(\textbf{x})  )  \ \bar{R}_{\delta} (\textbf{x}, \textbf{y}) d \mu_\textbf{y}.
\end{equation}
Utilizing again the point integral method, we discretize \eqref{numericalsoln1} into the following linear system
\begin{equation} \label{numericals}
\begin{cases}
\sum \limits_{j=1}^n L_{\delta}^{ij} (u_i-u_j) - \sum \limits_{k=1}^m G_{\delta}^{ik} v_k = f_{1\delta}^i+g_{1\delta}^i & i=1,2,...n; \\
\sum \limits_{j=1}^n D_{\delta}^{lj} u_j + \tilde{R}_{\delta}^l v_l=f_{2\delta}^l+g_{2\delta}^l  & l=1,2,...,m;
\end{cases}
\end{equation}
where in addition to \eqref{numericalsoln},
\begin{equation}
g_{1\delta}^i=-\sum \limits_{k=1}^m ((\textbf{p}_i-\q_k) \cdot \n_k) \Delta_{\partial \M} \ g (\q_k) \bar{R}_\delta(\textbf{p}_i,\q_k) L_k,
\end{equation}
\begin{equation}
g_{2\delta}^l=\sum \limits_{j=1}^n (2- \kappa_{\n} (\q_l) (\q_l-\textbf{p}_j) \cdot \n_l ) \bar{R}_\delta (\q_l, \textbf{p}_j) g (\q_l) A_j.
\end{equation}

Next, we study two numerical examples with non-homogeneous Dirichlet boundary. For convenience, we use the same manifold as example $1,2$ but with different exact solution $u$.
\subsubsection{2D Unit Hemisphere}
In the third example of this section, we let the manifold be part of the unit hemisphere
\begin{equation}
x^2+y^2+z^2=1, z \geq \frac{1}{2},
\end{equation}
 and let  the exact solution $u$ be $u (x,y,z)=x$, then we calculate
\begin{equation}
f(x,y,z)=2x, \ g(x,y,\frac{1}{2})=x.
\end{equation}
Still, we choose the kernel function $R$ to be \eqref{kernelr}, and implement the following iteration to solve \eqref{numericals}:

for t=4:11
\begin{enumerate}
\item  choose $\delta=1/t$, and let $n=2t^4+4t^2$, $m=4t^2$;
\item  repeat steps $2-7$ in the numerical example of section \ref{hmsphere} to obtain the point clouds $\{ \textbf{p}_i \}_{i=1}^n , \ \{ \q_k\}_{k=1}^m$,  and their corresponding weights $\mathcal{A}$ and $\LL$;
\item  calculate $f(\q_k), f(\textbf{p}_i), g(\q_k), \Delta_{\partial \M} g(\q_k), R_{\delta}, \bar{R}_{\delta}$ and $\overset{=}{R}_{\delta}$. In this specific example, $\kappa_{\n}(\q_k) =-\frac{\sqrt{3}}{3}$, $\n_k = <\frac{\sqrt{3}}{3} \q_k(1), \frac{\sqrt{3}}{3} \q_k(2), -\frac{\sqrt{3}}{2}>$, and $\Delta_{\partial \M} g(x,y, \frac{1}{2})=-\frac{4}{3}x$. We then complete the stiff matrix of \eqref{numericals};
\item we use function GMRES to solve the system \eqref{numericals} and obtain the solution vector $(U,V)$.
By comparing it with $u$, we output the value $e_2$ and $e_2^b$ defined in \eqref{ee2} and \eqref{ee2b}.
\end{enumerate}
end for \\

We then record in Table \ref{figure65} the value of $e_2$ and $e_{2b}$ in each iteration and their rate of change with respect to $\delta$ between each consecutive iterations.
 \begin{table}[htb] 
\begin{tabular}{|c|c|c|c|c|c|c|}
\hline
$\delta $  & n & m  & $e_2$ & rate of $e_2$ w.r.t. $\delta$ & $e_2^b$ & rate of $e_{2b}$ w.r.t. $\delta$ \\
\hline
0.250 & 576 & 64  &  0.0409 & N/A  &  0.0257 & N/A \\
\hline
0.200 & 1350 & 100  & 0.0299 & 1.4039 & 0.0150 &  2.4130 \\
\hline
0.167 &  2736 & 144  & 0.0188 &  2.5450 &  0.0107   & 1.8528  \\       
         \hline 
 0.143 & 4998 & 196   & 0.0132 & 2.2941 & 0.0085  & 1.4932 \\
\hline
0.125 & 8448 &256  & 0.0088 & 3.0365 &  0.0055   &  3.2600 \\
\hline
0.111 & 13446 & 324 & 0.0066 & 2.4425 &  0.0041   &  2.4941 \\
\hline
0.100 & 20400 & 400  & 0.0054  & 1.9046 &  0.0036  &  1.2344 \\       
         \hline 
0.091 & 29766  & 484    & 0.0043 &  2.3899 &  0.0029   & 2.2686 \\
\hline
\end{tabular}
 \caption{ \label{figure65} error of PIM and rate of convergence between each iteration in example 3.}
 \end{table}

Next, we plot the $8$ points $(\ln(\delta), \ln(e_2))$ from each iteration on the $2D$ rectangular coordinate system, and sketch the auxiliary line $y=2x-0.5$; besides, we plot the $8$ points $(\ln(\delta), \ln(e_{2b}))$ on the same plane and sketch the auxiliary line $y=2x-0.95$, to obtain Figure \ref{figure66}.

 \begin{figure}[htb] 
 \centering
\includegraphics[width=.82\textwidth]{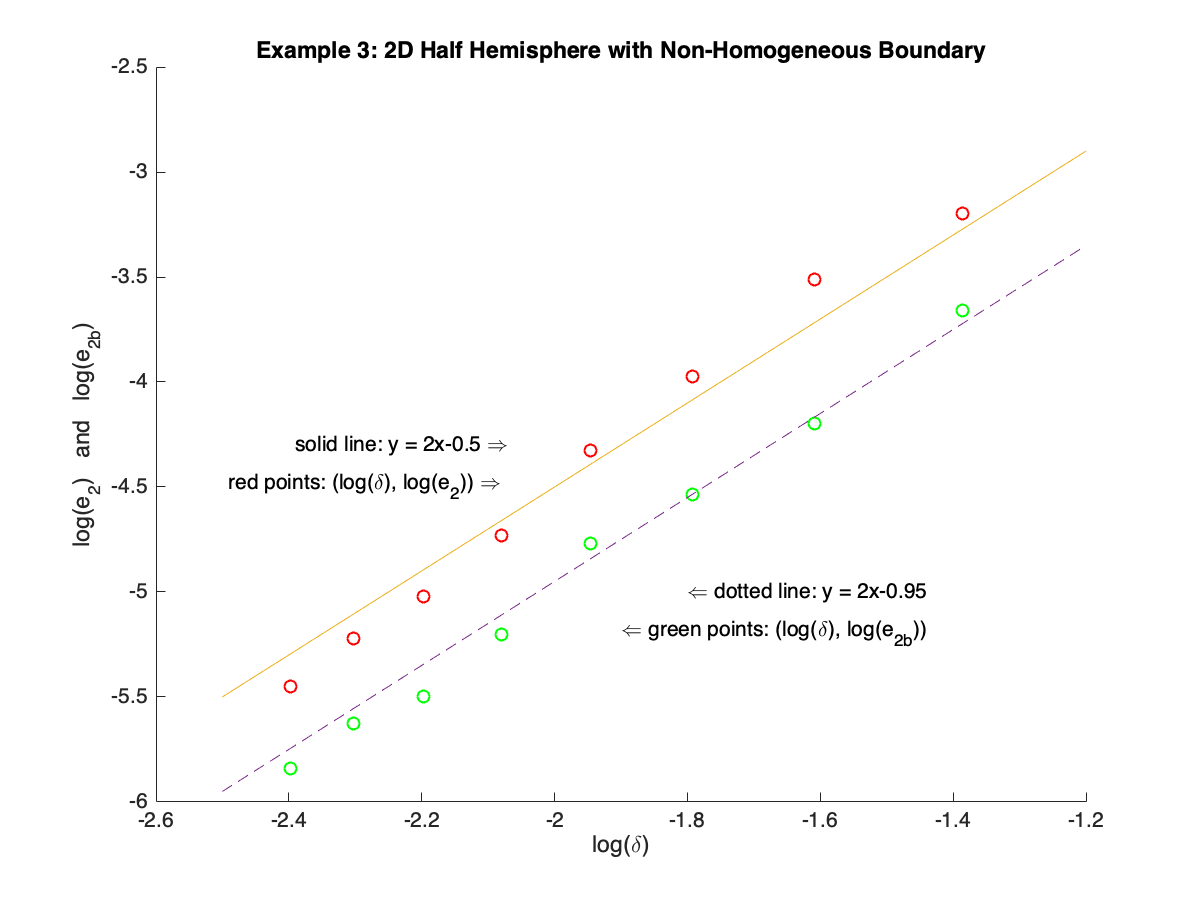}
 \caption{ \label{figure66} $l^2$ approximation error vs $\delta$ and their fitting lines in example 3.}
 \end{figure}

Figure \ref{figure66} indicates that $e_2, \ e_{2b}$ is almost linearly dependent on $\delta^2$. The perturbation is mainly from the randomness of the point cloud $\PP$ and $\QQ$.

\subsubsection{ 3D Manifold Embedded in $\mathbb{R}^4$ }
For the last example, we let the manifold $\M$ be 
\begin{equation}
x^2+y^2+z^2+w^2=1, w \geq 0,
\end{equation}
with its boundary $\partial \M$ be the unit ball $x^2+y^2+z^2=1, \ w=0$.
In the local Poisson problem \eqref{b01}, we let the exact solution $u$ be $u(x,y,z,w)=xy$. By the definition of $\Delta_{\M}$, we parametrize $\M$ and calculate 
\begin{equation}
f(x,y,z,w)=-\Delta_{\M} u=8xy, \qquad g(x,y,z,0)=xy, \qquad  \frac{\partial u}{\partial \n}(x,y,z,0)=0.
\end{equation}
Again, we choose the kernel function $R$ to be \eqref{kernelr}, and implement the following iteration to solve \eqref{numericals}:

for $t=2:6$
\begin{enumerate} 
\item choose $\delta=1/t$, and let $n=3t^6+4t^4$, $m=4t^4$;
\item repeat steps $2-8$ in the numerical example of section \ref{3dball} to obtain the point clouds $\{ \textbf{p}_i \}_{i=1}^n , \ \{ \q_k\}_{k=1}^m$,  and their corresponding weights $\mathcal{A}$ and $\LL$;
\item calculate $f(\q_k), f(\textbf{p}_i), R_{\delta}, \bar{R}_{\delta}$ and $\overset{=}{R}_{\delta}$. In this specific example, $\kappa_{\n}(\q_k) \equiv 0$, $\n_k \equiv <0, 0,0,-1>$, and $\Delta_{\partial \M} g(x,y,z,0)=-6xy$ . We then complete the stiff matrix of \eqref{numericalsoln};
\item we use function GMRES to solve the system \eqref{numericalsoln} and obtain the solution vector $(U,V)$.
By comparing it with $u$, we output the value $e_2$ and $e_2^b$ defined in \eqref{ee2} and \eqref{ee2b}.
\end{enumerate} 
end for \\

Again, we utilize Table \ref{figure67} to record $e_2, e_2^b$ and their rate of change with respect to $\delta$. In Figure \ref{figure68}, we plot the 10 points $(\ln(\delta), \ln(e_2))$ and $(\ln(\delta), \ln(e_2))$ on the 2D rectangular coordinate system and sketch 2 auxiliary lines $y=2x-0.3$, $y=2x-2.09$ to better explain the rate of convergence.
\begin{table}[htb] 
\begin{tabular}{|c|c|c|c|c|c|c|}
\hline
$\delta$ & n & m    & $e_2$ &  rate of $e_2$ w.r.t $\delta$ & $e_2^b$ & rate of $e_2^b$ w.r.t. $\delta$ \\
\hline
0.500 &  256 & 64 &   0.0323 &  N/A &  0.2131   &  N/A \\
\hline
0.333 &  2511 & 324 &   0.0126 &  2.3217 &  0.0911   &  2.0959 \\
\hline
0.250 &  13312 & 1024 &   0.0065 &  2.3008 &  0.0449  &  2.4594 \\
\hline
0.200 & 49375 & 2500 & 0.0044 &  1.7486 &  0.0255  &  2.5354  \\
\hline
0.167 & 145152 & 5184  & 0.0031 &  1.9208    &  0.0181   &  1.8800  \\       
         \hline 
\end{tabular}
\caption{ \label{figure67}  error of PIM and rate of convergence between each iteration in example 4.}
\label{diagram4}
\end{table}

\begin{figure}[htb]  
\centering
\includegraphics[width=.90\textwidth]{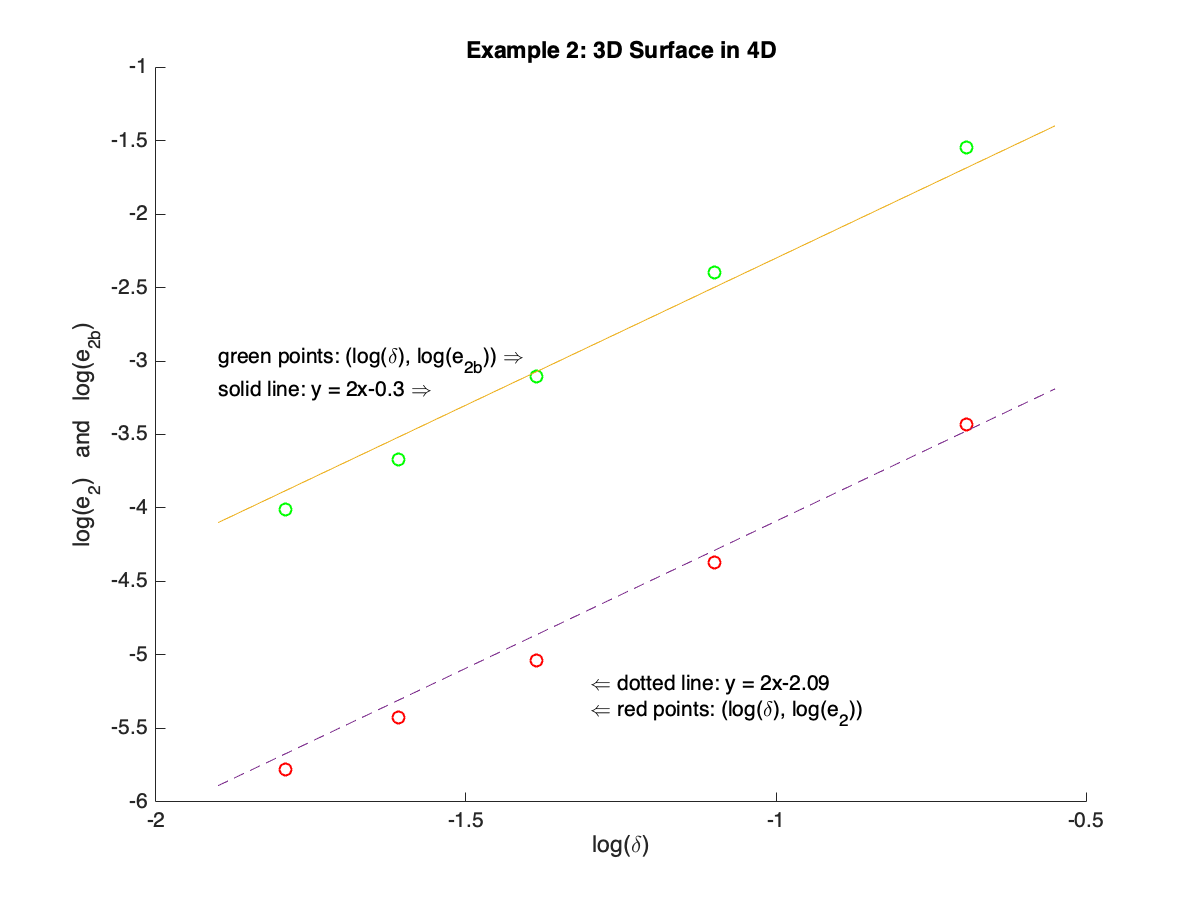}
\caption{ \label{figure68} $l^2$ approximation error vs $\delta$ and their fitting lines in example 4.}
 \end{figure}
 Figure \ref{figure68} indicates that $e_2, \ e_{2b}$ is almost linearly dependent on $\delta^2$ as well. 
 
Consequently, in the above two examples, the discrete solution of \eqref{numericals} generated by PIM converges to the exact solution of \eqref{bb01} in a rate of $\mathcal{O}(\delta^2)$ in the discrete $l^2$ norm,
which is $\mathcal{O}(h)$ where $h$ refers to the mesh size. This result indicates that the nonlocal Poisson model with non-homogeneous Dirichlet boundary can be solved by PIM as well, where the rate of convergence is preserved.
 
\subsection{Observation}
This section mainly introduces how PIM works on nonlocal manifold models and the approximation to its local counterpart. Compared to the manifold finite element method with piecewise linear elements, one advantage of PIM is that only local mesh is required so that we do not need a global mesh like the manifold FEM. Moreover, PIM can be efficiently applied when the explicit formulation of the manifold is not known except for a set of sample points, which is often occurred in data mining and machine learning models. 

Nevertheless, the quadrature rule we used in the point integral method is of low accuracy. If we have more information, such as the local mesh or local hyper-surface, we could use high order quadrature rule to improve the accuracy of the point integral method.

\section{Conclusion}
In this work, we have constructed a class of nonlocal models that approximates the Poisson equation on manifolds embedded in $\mathbb{R}^d$ under Dirichlet boundary.  Our calculation indicates that the convergence rate is $\mathcal{O} (\delta^2)$ in $H^1$ norm. To the author's best knowledge, even in the simpler case with Euclid domain of 3 dimensional or higher, all the previous studies have provided at most linear convergence rate. Having a Dirichlet-type constraint with second order convergence to the local limit in high dimensional manifold would be both mathematically and practically interesting.

Similar to the nonlocal approximation of Poisson models, the nonlocal approximation of some other types of PDEs are also of great interest. In our subsequent paper, we will introduce how to approximate the elliptic equation with discontinuous coefficients in high dimensional manifolds. Our future plan is to extend our results into a two dimensional polygonal domain where singularity appears near each vertex. The nonlocal approximation for Stokes equation with Dirichlet boundary will also be analyzed.

\section{Bibliography}
\bibliography{reference}%

\section{Appendix Section}

\subsection{Proof of Lemma \ref{lemma1}}

\begin{proof}
\begin{enumerate}
\item
we split this part into the following 5 inequalities
\begin{enumerate}
\item  $$C_1 \delta \geq \tilde{R}_\delta(\textbf{x}) \geq C_2 \delta, \qquad \forall \ a.e. \ \textbf{x} \in \M,$$
\item  $$ \int_{\partial \M}  n^2_{\delta}(\textbf{x}) \tilde{R}_{\delta}(\textbf{x})  d \tau_\textbf{x} 
+ \frac{1}{2\delta} \left \lVert q_{\delta} \right \rVert_{L^2(\partial \M)}^2
\geq C \left \lVert \bar{m}_{\delta} \right \rVert^2_{L^2(\partial \M)}  - \delta \left \lVert m_{\delta} \right \rVert^2_{L^2(\M)}, $$
\item  $$  \int_{\M}  \int_{\M}  (m_{\delta}(\textbf{x})-m_{\delta}(\textbf{y}))^2 \ {R}_{\delta} (\textbf{x}, \textbf{y}) d \mu_\textbf{x} d \mu_\textbf{y} \geq C \left \lVert m_{\delta} -\bar{m}_{\delta} \right \rVert^2_{L^2(\M)}, $$ 
\item $$\frac{1}{2 \delta^2} \int_{\M}  \int_{\M}  (m_{\delta}(\textbf{x})-m_{\delta}(\textbf{y}))^2 \ {R}_{\delta} (\textbf{x}, \textbf{y}) d \mu_\textbf{x} d \mu_\textbf{y} \geq C \left \lVert \nabla \bar{m}_{\delta} \right \rVert^2_{L^2(\M)},$$
\item $$ \left \lVert \nabla \bar{m}_{\delta} \right \rVert^2_{L^2(\M)}+ \left \lVert \bar{m}_{\delta} \right \rVert^2_{L^2(\partial \M)} \geq C \left \lVert  \bar{m}_{\delta} \right \rVert^2_{L^2(\M)},$$
\end{enumerate}
where first inequality implies
\begin{equation}
 \int_{\partial \M}  n^2_{\delta}(\textbf{x}) \tilde{R}_{\delta}(\textbf{x})  d \tau_{\textbf{x}} \geq C \delta \left \lVert  n_{\delta} \right \rVert^2_{L^2(\partial \M)},
 \end{equation}
and the direct sum of the last 4 inequalities illustrate
\begin{equation}
\begin{split}
\frac{1}{2 \delta^2} \int_{\M}  \int_{\M}  (m_{\delta}(\textbf{x})-m_{\delta}(\textbf{y}))^2 \ {R}_{\delta} (\textbf{x}, \textbf{y}) d \mu_\textbf{x} d \mu_\textbf{y} +   \int_{\partial \M}  n^2_{\delta}(\textbf{x}) \tilde{R}_{\delta}(\textbf{x})  d \tau_\textbf{x} 
\\   +\frac{1}{2\delta} \left \lVert  q_{\delta} \right \rVert^2_{L^2(\partial \M)} 
 \geq C \left \lVert  {m}_{\delta} \right \rVert^2_{L^2(\M)},
\end{split}
\end{equation}
we will then conclude \eqref{c50} according to \eqref{Bdelta}. Now let us prove these estimates in order.

\begin{enumerate}
\item Recall the definition of $\tilde{R}$ in \eqref{tilder},
\begin{equation}
\begin{split}
  \tilde{R}_{\delta}(\textbf{x}) = & 4 \delta^2 \int_{\partial \M}  \overset{=}{R}_{\delta} (\textbf{x}, \textbf{y}) d \tau_\textbf{y} -  \int_{\M} \kappa_{\n} (\textbf{x}) \ ((\textbf{x}-\textbf{y}) \cdot \n(\textbf{x}) )^2  \ \bar{R}_{\delta} (\textbf{x},\textbf{y}) d \mu_\textbf{y} .
\end{split}
\end{equation}
The second term is apparently $\mathcal{O}(\delta^2)$ and the first term is $\mathcal{O}(\delta)$. For small $\delta$, we have
\begin{equation}
\int_{\M} \kappa_{\n} (\textbf{x}) \ ((\textbf{x}-\textbf{y}) \cdot \n(\textbf{x}) )^2  \ \bar{R}_{\delta} (\textbf{x},\textbf{y}) d \mu_\textbf{y} \leq C \delta^2 \leq \delta^2 \int_{\partial \M}  \overset{=}{R}_{\delta} (\textbf{x}, \textbf{y}) d \tau_\textbf{y} ,
\end{equation}
hence we can conclude
\begin{equation}
3 \delta^2 \int_{\partial \M}  \overset{=}{R}_{\delta} (\textbf{x}, \textbf{y}) d \tau_\textbf{y} \leq  \tilde{R}_{\delta}(\textbf{x})  \leq 4 \delta^2 \int_{\partial \M}  \overset{=}{R}_{\delta} (\textbf{x}, \textbf{y}) d \tau_\textbf{y}.
\end{equation}
Due to our assumptions on $R$, we have $C_1 \delta \leq \delta^2 \int_{\partial \M}  \overset{=}{R}_{\delta} (\textbf{x}, \textbf{y}) d \tau_\textbf{y} \leq C_2 \delta$ for some constant $C_1, C_2>0$, it is clear that we can have both upper and lower bounds for $\tilde{R}_\delta$.

\item  We apply the inequality $\left \lVert a \right \rVert^2_{L^2(\M)} + \left \lVert b-a \right \rVert^2_{L^2(\M)} \geq C \left \lVert b \right \rVert^2_{L^2(\M)}$ to deduce
\begin{equation} \label{c53}
\begin{split}
& \int_{\partial \M}  n^2_{\delta}(\textbf{x}) \tilde{R}_{\delta}(\textbf{x})  d \tau_\textbf{x} + \frac{1}{2\delta} \left \lVert q_{\delta} \right \rVert_{L^2(\partial \M)}^2 \\
  = & \int_{\partial \M} \frac{1}{\tilde{R}_{\delta}(\textbf{x})} (q_{\delta}(\textbf{x})-\mathcal{D}_{\delta} m_{\delta}(\textbf{x}))^2 d \tau_\textbf{x}
+\frac{1}{2\delta} \left \lVert q_{\delta} \right \rVert_{L^2(\partial \M)}^2  \\
 \geq & \frac{C}{\delta} \int_{\partial \M} (q_{\delta}(\textbf{x})-\mathcal{D}_{\delta} m_{\delta}(\textbf{x}))^2 d \tau_\textbf{x} 
+ \frac{1}{2\delta} \int_{\partial \M}  q_{\delta}^2(\textbf{x}) d \tau_\textbf{x} 
 \geq \frac{C}{ \delta} \int_{\partial \M}    (\mathcal{D}_{\delta} m_{\delta}(\textbf{x}) )^2  d \tau_\textbf{x} \\
 \geq & C  \int_{\partial \M}    (    \int_{\M} m_{\delta} (\textbf{y}) \ (2-  \kappa_{\n} (\textbf{x}) \ (\textbf{x}-\textbf{y}) \cdot \n(\textbf{x})  )  \ \bar{R}_{\delta} (\textbf{x}, \textbf{y}) d \mu_\textbf{y}    )^2       d \tau_\textbf{x}.
\end{split}
\end{equation}
On the other hand, we have
\begin{equation} \label{c52}
\begin{split}
 & \int_{\partial \M}    (    \int_{\M} m_{\delta} (\textbf{y}) \  \kappa_{\n} (\textbf{x}) \ (\textbf{x}-\textbf{y}) \cdot \n(\textbf{x})   \ \bar{R}_{\delta} (\textbf{x}, \textbf{y}) d \mu_\textbf{y}    )^2       d \tau_\textbf{x} \\
& \leq C {\delta^2}  \int_{\partial \M}    (    \int_{\M} \vert m_{\delta} (\textbf{y}) \vert \  \kappa_{\n} (\textbf{x})   \ \bar{R}_{\delta} (\textbf{x}, \textbf{y}) d \mu_\textbf{y}    )^2       d \tau_\textbf{x}  \\
& \leq C {\delta^2}  \int_{\partial \M}    \kappa_{\n}^2(\textbf{x}) \  ( \int_{\M} \vert m_{\delta} (\textbf{y}) \vert^2   \ \bar{R}_{\delta} (\textbf{x}, \textbf{y}) d \mu_\textbf{y}  ) \ (     \int_{\M}    \ \bar{R}_{\delta} (\textbf{x}, \textbf{y}) d \mu_\textbf{y}   )    \  d \tau_\textbf{x}  \\
& \leq C  {\delta^2} \int_{\M} (  \int_{\partial \M}   \kappa_{\n}^2(\textbf{x}) \bar{R}_{\delta} (\textbf{x}, \textbf{y}) d \tau_\textbf{x} ) \  \vert m_{\delta} (\textbf{y}) \vert^2   \  d \mu_\textbf{y}      \leq  C\delta \left \lVert m_{\delta} \right \rVert^2_{L^2(\M)},
\end{split}
\end{equation}
we apply again the inequality $\left \lVert a \right \rVert^2_{L^2(\M)} + \left \lVert b-a \right \rVert^2_{L^2(\M)} \geq C \left \lVert b \right \rVert^2_{L^2(\M)}$ into \eqref{c52} to discover
\begin{equation} \label{c54}
\begin{split}
&  \int_{\partial \M}    (    \int_{\M} m_{\delta} (\textbf{y}) \ (2- \kappa_{\n} (\textbf{x}) \ (\textbf{x}-\textbf{y}) \cdot \n(\textbf{x})  )  \ \bar{R}_{\delta} (\textbf{x}, \textbf{y}) d \mu_\textbf{y}    )^2       d \tau_\textbf{x}   
+ \delta \left \lVert m_{\delta} \right \rVert^2_{L^2(\M)} \\
& \geq C \int_{\partial \M}    (   2 \int_{\M} m_{\delta} (\textbf{y})  \ \bar{R}_{\delta} (\textbf{x}, \textbf{y}) d \mu_\textbf{y}    )^2       d \tau_\textbf{x}
 \geq  C \int_{\partial \M}       \bar{m}^2_\delta (\textbf{x})      d \tau_\textbf{x} 
 = C \left \lVert \bar{m}_{\delta} \right \rVert^2_{L^2(\partial \M)}.
\end{split}
\end{equation}
Hence we combine \eqref{c53} and \eqref{c54} to conclude
\begin{equation}
 \int_{\partial \M}  n^2_{\delta}(\textbf{x}) \tilde{R}_{\delta}(\textbf{x})  d \tau_\textbf{x} + \frac{1}{2\delta} \left \lVert q_{\delta} \right \rVert_{L^2(\partial \M)}^2 + \delta \left \lVert m_{\delta} \right \rVert^2_{L^2(\M)}
 \geq C \left \lVert \bar{m}_{\delta} \right \rVert^2_{L^2(\partial \M)}.
\end{equation}

\item We can calculate
\begin{equation} \label{lm521}
\begin{split}
& \left \lVert \bar{m}_{\delta}-m_{\delta}  \right \rVert^2_{L^2(\M)}= \int_{\M} \big( \int_{\M}  \frac{1}{\bar{\omega}_{\delta}(\textbf{x})}  ( m_{\delta}(\textbf{x}) - m_{\delta} (\textbf{y}) ) \bar{R}_{\delta} (\textbf{x}, \textbf{y}) d \mu_\textbf{y} \big)^2 \ d \mu_\textbf{x} \\
& \leq C    \int_{\M} \big( \int_{\M}  ( m_{\delta}(\textbf{x}) - m_{\delta} (\textbf{y}) ) \bar{R}_{\delta} (\textbf{x}, \textbf{y}) d \mu_\textbf{y} \big)^2 \ d \mu_\textbf{x} \\
& \leq C \int_{\M} \ ( \int_{\M} \bar{R}_{\delta}(\textbf{x}, \textbf{y}) \ d \mu_\textbf{y} \ ) \ ( \int_{\M} \bar{R}_{\delta} (\textbf{x}, \textbf{y}) (m_{\delta}(\textbf{x})-m_{\delta}(\textbf{y}))^2 d \mu_\textbf{y} ) \ d \mu_\textbf{x} \\
& \leq C \int_{\M} \int_{\M} \bar{R}_{\delta} (\textbf{x}, \textbf{y}) (m_{\delta}(\textbf{x})-m_{\delta}(\textbf{y}))^2 d \mu_\textbf{y} d \mu_\textbf{x}  \\
& \leq C \int_{\M} \int_{\M} R_{\delta} (\textbf{x}, \textbf{y}) (m_{\delta}(\textbf{x})-m_{\delta}(\textbf{y}))^2 d \mu_\textbf{y} d \mu_\textbf{x}
\leq C \delta^2 B_{\delta}[m_{\delta}, n_{\delta};  m_{\delta}, n_{\delta}].
\end{split}
\end{equation}
\item This is exactly the equation \eqref{c51}.
\item This is the manifold version of Poincare inequality for $\overset{=}{m}_\delta \in H^1(\M)$.
\end{enumerate}

\item As usual, we split the proof into the following steps
\begin{enumerate}
\item 
\begin{equation}
\left \lVert \nabla \hat{m}_{\delta} \right \rVert^2_{L^2(\M)} \leq \frac{C}{2 \delta^2} \int_{\M}  \int_{\M}  (m_{\delta}(\textbf{x})-m_{\delta}(\textbf{y}))^2 \ {R}_{\delta} (\textbf{x}, \textbf{y}) d \mu_\textbf{x} d \mu_\textbf{y} , 
\end{equation}
\item 
\begin{equation}
 \left \lVert \nabla ( m_{\delta}(\textbf{x})-\hat{m}_{\delta}(\textbf{x}) ) \right \rVert_{L^2(\M)} \leq C (\delta^2 F(\delta)  \left \lVert  p_0 \right \rVert_{H^\beta(\M)} + \delta^{\frac{1}{2}}\left \lVert n_{\delta} \right \rVert_{L^2(\partial \M)}  ), 
 \end{equation}
\item  
\begin{equation}
\left \lVert {m}_{\delta} \right \rVert^2_{L^2(\M)} + \delta \left \lVert n_{\delta} \right \rVert_{L^2(\partial \M)}^2  \leq C(B_{\delta}[m_{\delta}, n_{\delta}; m_{\delta}, n_{\delta}] + \frac{1}{ \delta }  \left \lVert q_{\delta} \right \rVert^2_{L^2(\partial \M)}), 
\end{equation}
\item 
\begin{equation}
\begin{split}
 C \ B_{\delta}[m_{\delta}, n_{\delta}; m_{\delta}, n_{\delta}] \leq &
\frac{1}{2} \left \lVert m_{\delta} \right \rVert_{H^1(\M)}^2+  \frac{\delta}{2} \left \lVert n_{\delta} \right \rVert_{L^2(\partial \M)}^2 \\
& +   C_1 \ ( \ G^2(\delta)  \left \lVert  p_0 \right \rVert^2_{H^\beta(\M)}  + \frac{1}{\delta} \left \lVert q_{\delta} \right \rVert^2_{L^2(\partial \M)} ), 
\end{split}
\end{equation}
\end{enumerate}
where the first $3$ inequalities imply
\begin{equation} \label{return1}
\begin{split}
\left \lVert {m}_{\delta} \right \rVert^2_{H^1(\M)} + \delta \left \lVert n_{\delta} \right \rVert_{L^2(\partial \M)}^2  \leq C(B_{\delta}[m_{\delta}, n_{\delta}; m_{\delta}, n_{\delta}] + \frac{1}{ \delta }  \left \lVert q_{\delta} \right \rVert^2_{L^2(\partial \M)} + \\ \delta^4 F^2(\delta)  \left \lVert  p_0 \right \rVert^2_{H^\beta(\M)} ),
\end{split}
\end{equation}
we will then deduce \eqref{c55} by combining the $4^{th}$ inequality and \eqref{return1}. Now let us prove them in order.
\begin{enumerate}
\item This is exactly the inequality \eqref{c12}.
\item This inequality is derived from the equation $\mathcal{L}_{\delta} m_{\delta}(\textbf{x}) - \mathcal{G}_{\delta} n_{\delta}(\textbf{x})
=  p_{\delta}(\textbf{x})$, or in other words,
\begin{equation}
\begin{split}
\frac{1}{ \delta^2} \int_{\M} (m_{\delta}(\textbf{x})-m_{\delta}(\textbf{y})) \ {R}_{\delta} (\textbf{x}, \textbf{y}) d \mu_\textbf{y} -
 \int_{\partial \M} n_{\delta} (\textbf{y}) \ (2+ \kappa_{\n} (\textbf{y}) \ (\textbf{x}-\textbf{y}) \cdot \n (\textbf{y}) ) \\ \bar{R}_{\delta} (\textbf{x}, \textbf{y})  d \tau_\textbf{y}
 =p_{\delta}(\textbf{x}).
 \end{split}
\end{equation}
Recall the definition of $\bar{m}_\delta$, we have
\begin{equation}
\begin{split}
\frac{1}{ \delta^2} \ \omega_{\delta}(\textbf{x}) & ( m_{\delta}(\textbf{x})-\hat{m}_{\delta}(\textbf{x}) ) 
-\int_{\partial \M} n_{\delta} (\textbf{y}) \ (2+ \kappa_{\n} (\textbf{y}) \ (\textbf{x}-\textbf{y}) \cdot \n (\textbf{y}) ) \ \bar{R}_{\delta} (\textbf{x}, \textbf{y})  d \tau_\textbf{y} \\
  & =p_{\delta}(\textbf{x}),
\end{split}
\end{equation}
Hence we obtain
\begin{equation} \label{c13}
\begin{split}
& \left \lVert  \nabla ( m_{\delta}-\hat{m}_{\delta} ) \right \rVert_{L^2(\M)} \leq  \delta^2 \left \lVert \nabla  \frac{p_{\delta}(\textbf{x})}{\omega_{\delta}(\textbf{x})} \right \rVert_{L^2(\M)} \\
& +  \delta^2 \left \lVert \nabla \ \int_{\partial \M} \ \frac{1}{\omega_{\delta}(\textbf{x})} \ n_{\delta} (\textbf{y}) \ (2+ \kappa_{\n} (\textbf{y}) \ (\textbf{x}-\textbf{y}) \cdot \n (\textbf{y}) ) \ \bar{R}_{\delta} (\textbf{x}, \textbf{y})  d \tau_\textbf{y} \right \rVert_{{L}_{\textbf{x}}^2(\M)}.
\end{split}
\end{equation}

The first term of \eqref{c13} can be controlled by
\begin{equation} \label{c14}
\begin{split}
& \left \lVert \nabla  \frac{p_{\delta}(\textbf{x})}{\omega_{\delta}(\textbf{x})} \right \rVert_{{L}^2(\M)}=\left \lVert \frac{ \omega_{\delta}(\textbf{x}) \ \nabla p_{\delta}(\textbf{x})- p_{\delta}(\textbf{x}) \nabla \omega_{\delta}(\textbf{x}) }{ \omega^2_{\delta}(\textbf{x})} \right \rVert_{L^2(\M)} \\
& \leq 2 \left \lVert \frac{1}{\omega_{\delta}(\textbf{x})} \ \nabla p_{\delta}(\textbf{x}) \right \rVert_{{L}^2(\M)} + 2 \left \lVert p_{\delta}(\textbf{x}) \frac{  \nabla \omega_{\delta}(\textbf{x}) }{\omega^2_{\delta}(\textbf{x})} \right \rVert_{{L}^2(\M)} \\
& \leq C (\left \lVert  \nabla p_{\delta}(\textbf{x}) \right \rVert_{{L}^2(\M)} + \frac{1}{\delta} \left \lVert p_{\delta}(\textbf{x})  \right \rVert_{{L}^2(\M)}) \leq C \ F(\delta)  \left \lVert  p_0 \right \rVert_{H^\beta(\M)},
\end{split}
\end{equation}
where the second inequality results from the fact that $C_1 \leq \omega_{\delta}(\textbf{x}) \leq C_2$ and 
\begin{equation}
\begin{split}
& \vert \nabla \omega_{\delta}(\textbf{x}) \vert= \vert \int_{\M} \nabla_{\M}^{\textbf{x}} \ R_{\delta}(\textbf{x},\textbf{y}) d \mu_\textbf{y} \ \vert= \vert\int_{\M} \nabla_\textbf{y} \ R_{\delta}(\textbf{x},\textbf{y}) d \mu_\textbf{y} \ \vert \\
& =\vert \int_{\partial \M} R_{\delta}(\textbf{x},\textbf{y}) \ \n(\textbf{y}) \ d \tau_\textbf{y} \ \vert \leq  \int_{\partial \M} R_{\delta}(\textbf{x},\textbf{y})  d \tau_\textbf{y} \leq C \frac{1}{\delta}, \qquad \forall \ \textbf{x} \in \M.
\end{split}
\end{equation}
The control on the second term of \eqref{c13} is more complicated in calculation. Similar to \eqref{c14}, we have
\begin{equation} \label{c15}
\begin{split}
& \left \lVert \nabla \ \int_{\partial \M} \ \frac{1}{\omega_{\delta}(\textbf{x})} \ n_{\delta} (\textbf{y}) \ (2+ \kappa_{\n} (\textbf{y}) \ (\textbf{x}-\textbf{y}) \cdot \n (\textbf{y}) ) \ \bar{R}_{\delta} (\textbf{x}, \textbf{y})  d \tau_\textbf{y} \right \rVert_{{L}_{\textbf{x}}^2(\M)} \\
& \leq C \ ( \left \lVert \ \nabla \ \int_{\partial \M}  \ n_{\delta} (\textbf{y}) \ (2+\kappa_{\n} (\textbf{y}) \ (\textbf{x}-\textbf{y}) \cdot \n (\textbf{y}) ) \ \bar{R}_{\delta} (\textbf{x}, \textbf{y})  d \tau_\textbf{y} \ \right \rVert_{{L}_{\textbf{x}}^2(\M)} \\
& +\frac{1}{\delta} \left \lVert \ \int_{\partial \M}  \ n_{\delta} (\textbf{y}) \ (2+ \kappa_{\n} (\textbf{y}) \ (\textbf{x}-\textbf{y}) \cdot \n (\textbf{y}) ) \ \bar{R}_{\delta} (\textbf{x}, \textbf{y})  d \tau_\textbf{y} \ \right \rVert_{{L}_{\textbf{x}}^2(\M)} ) \\
& \leq C \ ( \ \left \lVert \ \int_{\partial \M} 3 \ \vert n_{\delta}(\textbf{y}) \vert \ \frac{1}{2\delta^2} \vert\textbf{x}-\textbf{y}\vert \ R_{\delta}(\textbf{x},\textbf{y})  d \tau_\textbf{y} \ \right \rVert_{{L}_{\textbf{x}}^2(\M)} \\
& + \left \lVert \ \int_{\partial \M}  \ \vert \ n_{\delta} (\textbf{y}) \vert \ \bar{R}_{\delta} (\textbf{x}, \textbf{y})  d \tau_\textbf{y} \ \right \rVert_{L^2(\M)} \\
& + \ \frac{1}{\delta} \ \left \lVert \ \int_{\partial \M}  \ 3 \ \vert  n_{\delta} (\textbf{y})  \vert \ \bar{R}_{\delta} (\textbf{x}, \textbf{y})  d \tau_\textbf{y} \ \right \rVert_{{L}_{\textbf{x}}^2(\M)} ) \\
& \leq \frac{C}{\delta}  \ \left \lVert \ \int_{\partial \M}  \vert n_{\delta}(\textbf{y}) \vert \ R_{\delta}(\textbf{x},\textbf{y})  d \tau_\textbf{y} \ \right \rVert_{{L}_{\textbf{x}}^2(\M)} \\
& \leq \frac{C}{\delta} ( \int_{\M}  (\int_{\partial \M}  n^2_{\delta}(\textbf{y})  R_{\delta}(\textbf{x},\textbf{y}) \ d \tau_\textbf{y} )( \int_{\partial \M} R_{\delta}(\textbf{x},\textbf{y}) \ d \tau_\textbf{y} ) d \mu_\textbf{x} )^{\frac{1}{2}} \\
& \leq \frac{C}{\delta}   (\int_{\partial \M}   \int_{\M} \frac{1}{\delta} \ n^2_{\delta}(\textbf{y}) \ R_{\delta}(\textbf{x},\textbf{y}) \ d \mu_\textbf{x} \  d \tau_\textbf{y}     )^{\frac{1}{2}}  \leq C \delta^{-\frac{3}{2}} \left \lVert n_{\delta} \right \rVert_{L^2(\partial \M)}.
\end{split}
\end{equation}
We therefore conclude \eqref{c13}, \eqref{c14} and \eqref{c15} to discover
\begin{equation}
\left \lVert \nabla ( m_{\delta}(\textbf{x})-\hat{m}_{\delta}(\textbf{x}) ) \right \rVert_{L^2(\M)} \leq C (\delta^2 F(\delta)  \left \lVert  p_0 \right \rVert_{H^\beta(\M)} + \delta^{\frac{1}{2}}\left \lVert n_{\delta} \right \rVert_{L^2(\partial \M)}  ),
\end{equation}
\item This is exactly the first part of the lemma.

\item
In fact, the bilinear form of the system $\eqref{c16}$ gives 
\begin{equation}  \label{lm534}
\begin{split}
2C \ B_{\delta} & [m_{\delta}, n_{\delta}; m_{\delta}, n_{\delta}] =2C \int_{\M} m_{\delta}(\textbf{x}) p_{\delta}(\textbf{x}) d \mu_\textbf{x}+2 C \int_{\partial \M} n_{\delta}(\textbf{x}) q_{\delta}(\textbf{x}) d \tau_\textbf{x} \\
 \leq  & 2 C \ G(\delta)  ( \left \lVert  m_{\delta} \right \rVert_{H^1(\M)}   + \left \lVert  \bar{m}_{\delta} \right \rVert_{H^1(\M)}    + \left \lVert  \overset{=}{m}_{\delta} \right \rVert_{H^1(\M)}    ) \left \lVert  p_0 \right \rVert_{H^\beta(\M)} \\
  & +2 C \left \lVert  n_{\delta} \right \rVert_{L^2(\partial \M)} \left \lVert  q_{\delta} \right \rVert_{L^2(\partial \M)} .
\end{split}
\end{equation}

Similar as the equation \eqref{c15}, we follow the calculation of \eqref{lm521} to obtain
\begin{equation}  \label{lm531}
\begin{split}
 \left \lVert \bar{m}_{\delta}-m_{\delta}  \right \rVert^2_{H^1(\M)} &= \int_{\M} \big( \int_{\M}  \frac{1}{\bar{\omega}_{\delta}(\textbf{x})}  ( m_{\delta}(\textbf{x}) - m_{\delta} (\textbf{y}) ) \bar{R}_{\delta} (\textbf{x}, \textbf{y}) d \mu_\textbf{y} \big)^2 \ d \mu_\textbf{x} \\
 & +   \int_{\M} \big( \int_{\M} \nabla^{\textbf{x}}_{\M}  \frac{1}{\bar{\omega}_{\delta}(\textbf{x})}  ( m_{\delta}(\textbf{x}) - m_{\delta} (\textbf{y}) ) \bar{R}_{\delta} (\textbf{x}, \textbf{y}) d \mu_\textbf{y} \big)^2 \ d \mu_\textbf{x} \\
 \leq & C (\delta^2 B_{\delta}[m_{\delta}, n_{\delta};  m_{\delta}, n_{\delta}] + B_{\delta}[m_{\delta}, n_{\delta};  m_{\delta}, n_{\delta}] )  \\
 \leq & C  B_{\delta}[m_{\delta}, n_{\delta};  m_{\delta}, n_{\delta}] .
\end{split}
\end{equation}
By substituting $\bar{R}_{\delta} $ by $\overset{=}{R}$ in \eqref{lm531}, we can obtain the following property for $\overset{=}{m}_{\delta}$:
\begin{equation}
  \left \lVert \overset{=}{m}_{\delta}-m_{\delta}  \right \rVert^2_{H^1(\M)}  \leq C  B_{\delta}[m_{\delta}, n_{\delta};  m_{\delta}, n_{\delta}] .
\end{equation}
This indicates
\begin{equation} \label{lm532}
\begin{split}
& 2 C \ G(\delta)  ( \left \lVert  m_{\delta} \right \rVert_{H^1(\M)}   + \left \lVert  \bar{m}_{\delta} \right \rVert_{H^1(\M)}    + \left \lVert  \overset{=}{m}_{\delta} \right \rVert_{H^1(\M)}    ) \left \lVert  p_0 \right \rVert_{H^\beta(\M)}  \\
\leq &  2 C \ G(\delta)  ( 3 \left \lVert  m_{\delta} \right \rVert_{H^1(\M)}  +C_0    B_{\delta}[m_{\delta}, n_{\delta};  m_{\delta}, n_{\delta}]    ) \left \lVert  p_0 \right \rVert_{H^\beta(\M)} \\
\leq &  \frac{1}{2} \left \lVert  m_{\delta} \right \rVert^2_{H^1(\M)} + C \ B_{\delta}  [m_{\delta}, n_{\delta}; m_{\delta}, n_{\delta}]   +( 18 C^2  + C   C_0^2)  G^2(\delta) \left \lVert  p_0 \right \rVert_{H^\beta(\M)} ,
\end{split}
\end{equation}
On the other hand, we have
\begin{equation} \label{lm533}
2 C \left \lVert  n_{\delta} \right \rVert_{L^2(\partial \M)} \left \lVert  q_{\delta} \right \rVert_{L^2(\partial \M)}
\leq \frac{\delta}{2}  \left \lVert  n_{\delta} \right \rVert_{L^2(\partial \M)} \ + \frac{2C^2}{\delta}  \left \lVert  q_{\delta} \right \rVert_{L^2(\partial \M)} ,
\end{equation}

We then combine the equations \eqref{lm534} \eqref{lm532} \eqref{lm533} to obtain
\begin{equation}
\begin{split}
C \ B_{\delta} & [m_{\delta}, n_{\delta}; m_{\delta}, n_{\delta}] \leq  \frac{1}{2} \left \lVert  m_{\delta} \right \rVert^2_{H^1(\M)} 
+ \frac{\delta}{2}  \left \lVert  n_{\delta} \right \rVert_{L^2(\partial \M)} \\
& + ( 18 C^2 + C    C_0^2) G^2(\delta) \left \lVert  p_0 \right \rVert_{H^\beta(\M)} + \frac{2C^2}{\delta}  \left \lVert  q_{\delta} \right \rVert_{L^2(\partial \M)}.
\end{split}
\end{equation}

Hence we have completed our proof.

\end{enumerate}
\end{enumerate}
\end{proof}

\subsection{Proof of (b) in Page 9}
\begin{proof}

For any $u_{\delta}, w_{\delta} \in L^2(\M)$, we can calculate
\begin{equation} 
\begin{split}
 \int_{\M} & ( \mathcal{L}_{\delta}u_{\delta} (\textbf{x})+(\mathcal{G}_{\delta} \frac{\mathcal{D}_{\delta} u_{\delta} }{  \tilde{R}_{\delta} } )(\textbf{x})  ) w_{\delta}(\textbf{x}) d \mu_\textbf{x} 
=\int_{\M} \int_{\M} w_{\delta} (\textbf{x}) (u_{\delta} (\textbf{x})-u_{\delta}(\textbf{y})) R_{\delta}(\textbf{x},\textbf{y})  d \mu_\textbf{y} d \mu_\textbf{x}  \\ 
 + & \int_{\M}  \int_{\partial \M}  \int_{\M} u (\s) \ (2 + \kappa(\textbf{y}) \ (\textbf{y}-\s) \cdot \n(\textbf{y})  )  \ \bar{R}_{\delta} (\textbf{y}, \s) d \s  \\ &    (2- \kappa(\textbf{y}) \ (\textbf{x}-\textbf{y}) \cdot \n (\textbf{y}) ) \ \bar{R}_{\delta} (\textbf{x}, \textbf{y})  d \tau_\textbf{y}  \ w_{\delta}(\textbf{x}) d \mu_\textbf{x},
\end{split}
\end{equation}
here
\begin{equation}
\begin{split}
& \big\vert \int_{\M} \int_{\M} w_{\delta} (\textbf{x}) (u_{\delta} (\textbf{x})-u_{\delta}(\textbf{y})) R_{\delta}(\textbf{x},\textbf{y})  d \mu_\textbf{y} d \mu_\textbf{x}  \big\vert \\
& \leq C_{\delta} (\int_{\M} \int_{\M} \vert w_{\delta} (\textbf{x}) u_{\delta} (\textbf{x})\vert d \mu_\textbf{y} d \mu_\textbf{x} + \int_{\M} \int_{\M} \vert w_{\delta} (\textbf{x}) u_{\delta} (\textbf{y})\vert d \mu_\textbf{y} d \mu_\textbf{x} ) \\
 & \leq C_{\delta}( \left \lVert w_{\delta} \right \rVert_{L^2(\M)} \left \lVert u_{\delta} \right \rVert_{L^2(\M)} +  \left \lVert w_{\delta} \right \rVert_{L^1(\M)} \left \lVert u_{\delta} \right \rVert_{L^1(\M)}  ) \leq C_{\delta} \left \lVert w_{\delta} \right \rVert_{L^2(\M)} \left \lVert u_{\delta} \right \rVert_{L^2(\M)} ;
 \end{split}
\end{equation}
and
\begin{equation}
\begin{split}
& \big\vert \int_{\M}  \int_{\partial \M}  \int_{\M} u_{\delta} (\s) \ (2 -\kappa(\textbf{y}) \ (\textbf{y}-\s) \cdot \n(\textbf{y})  )  \ \bar{R}_{\delta} (\textbf{y}, \s) d \s     \\ & (2+\kappa(\textbf{y}) \ (\textbf{x}-\textbf{y}) \cdot \n (\textbf{y}) ) \ \bar{R}_{\delta} (\textbf{x}, \textbf{y})  d \tau_\textbf{y}  \ w_{\delta}(\textbf{x}) d \mu_\textbf{x} \big\vert \\
& \leq C_{\delta} \int_{\M}  \int_{\partial \M}  \int_{\M} \vert u_{\delta} (\s) w_{\delta}(\textbf{x}) \vert  \ d \mu_{\s} d \tau_\textbf{y} d \mu_\textbf{x} 
\leq C_{\delta} \left \lVert w_{\delta} \right \rVert_{L^1(\M)} \left \lVert u_{\delta} \right \rVert_{L^1(\M)}  \\ &
\leq C_{\delta} \left \lVert w_{\delta} \right \rVert_{L^2(\M)} \left \lVert u_{\delta} \right \rVert_{L^2(\M)} .
\end{split}
\end{equation}
The above 2 inequalities implies that 
\begin{equation} \label{fix4}
\int_{\M} ( \mathcal{L}_{\delta}u_{\delta} (\textbf{x}) + (\mathcal{G}_{\delta} \frac{\mathcal{D}_{\delta} u_{\delta} }{  \tilde{R}_{\delta} } )(\textbf{x})  ) w_{\delta}(\textbf{x}) d \mu_\textbf{x}  \leq C_{\delta} \left \lVert w_{\delta} \right \rVert_{L^2(\M)} \left \lVert u_{\delta} \right \rVert_{L^2(\M)},
 \end{equation}
  where $C_{\delta}$ is a constant depend on $\delta$ and independent on $u_{\delta}$ and $w_{\delta}$.
  \end{proof}

\subsection{Proof of (c) in Page 9}
\begin{proof}
We first split the right hand side into
\begin{equation}
\begin{split}
& \int_{\M} w_{\delta}(\textbf{x}) \mathcal{P}_{\delta} f(\textbf{x}) d \mu_\textbf{x} = \int_{\M} w_{\delta}(\textbf{x}) \int_{\M} f(\textbf{y}) \ \bar{R}_{\delta} (\textbf{x}, \textbf{y}) d \mu_\textbf{y} d \mu_\textbf{x}  \\ &
-\int_{\M} w_{\delta}(\textbf{x}) \int_{\partial \M} ((\textbf{x}-\textbf{y}) \cdot  \n(\textbf{y}) )  \ f(\textbf{y}) \ \bar{R}_{\delta} (\textbf{x}, \textbf{y}) d \tau_\textbf{y} d \mu_\textbf{x}, \\
\end{split}
\end{equation}
and we can calculate
\begin{equation}   \label{sm1}
\begin{split}
& \int_{\M} w_{\delta}(\textbf{x}) \int_{\M} f(\textbf{y}) \ \bar{R}_{\delta} (\textbf{x}, \textbf{y}) d \mu_\textbf{y} d \mu_\textbf{x}   
\leq \Big[ \int_{\M} w_{\delta}^2(\textbf{x}) d \mu_\textbf{x} \ \int_{\M} ( \int_{\M} f(\textbf{y}) \bar{R}_{\delta}(\textbf{x},\textbf{y}) d \mu_\textbf{y} )^2 d \mu_\textbf{x} \Big]^{\frac{1}{2}} \\
& \leq \Big[ \int_{\M} w_{\delta}^2(\textbf{x}) d \mu_\textbf{x} \ \int_{\M} ( \int_{\M} f^2(\textbf{y}) \bar{R}_{\delta}(\textbf{x},\textbf{y}) d \mu_\textbf{y} \ \int_{\M} \bar{R}_{\delta}(\textbf{x},\textbf{y}) d \mu_\textbf{y} ) d \mu_\textbf{x} \ \Big]^{\frac{1}{2}}  \\
& \leq \Big[ \int_{\M} w_{\delta}^2(\textbf{x}) d \mu_\textbf{x} \ \int_{\M}  \int_{\M}  f^2(\textbf{y})  \bar{R}_{\delta}(\textbf{x},\textbf{y})  d \mu_\textbf{y}  d \mu_\textbf{x}  \ \Big]^{\frac{1}{2}}  \\
& \leq  \Big[ \int_{\M} w_{\delta}^2(\textbf{x}) d \mu_\textbf{x} \ \int_{\M}    f^2(\textbf{y})  d \mu_\textbf{y}   \ \Big]^{\frac{1}{2}} 
\leq \left \lVert f \right \rVert_{L^{2}( \M)} \left \lVert w_{\delta} \right \rVert_{L^2(\M)} 
\leq  \left \lVert f \right \rVert_{H^1( \M)} \left \lVert w_{\delta} \right \rVert_{L^2(\M)},
\end{split}
\end{equation}

\begin{equation} \label{sm2}
\begin{split}
& \int_{\M} w_{\delta}(\textbf{x}) \int_{\partial \M} ((\textbf{x}-\textbf{y}) \cdot  \n(\textbf{y}) )  \ f(\textbf{y}) \ \bar{R}_{\delta} (\textbf{x}, \textbf{y}) d \tau_\textbf{y} d \mu_\textbf{x}  \\
& \leq \delta \int_{\M} \vert w_{\delta}(\textbf{x})\vert \int_{\partial \M}  \ \vert f(\textbf{y})\vert \ \bar{R}_{\delta} (\textbf{x}, \textbf{y}) d \tau_\textbf{y} d \mu_\textbf{x} \\
 & \leq \delta \Big[ \int_{\M} w_{\delta}^2(\textbf{x}) d \mu_\textbf{x} \ \int_{\M} ( \int_{\partial \M} \vert f(\textbf{y})\vert \bar{R}_{\delta}(\textbf{x},\textbf{y}) d \tau_\textbf{y} )^2 d \mu_\textbf{x} \Big]^{\frac{1}{2}} \\
& \leq \delta \Big[ \int_{\M} w_{\delta}^2(\textbf{x}) d \mu_\textbf{x} \ \int_{\M} ( \int_{\partial \M} f^2(\textbf{y}) \bar{R}_{\delta}(\textbf{x},\textbf{y}) d \tau_\textbf{y} \ \int_{\partial \M} \bar{R}_{\delta}(\textbf{x},\textbf{y}) d \mu_\textbf{y} ) d \mu_\textbf{x} \ \Big]^{\frac{1}{2}}  \\
& \leq \delta^{\frac{1}{2}} \Big[ \int_{\M} w_{\delta}^2(\textbf{x}) d \mu_\textbf{x} \ \int_{\M}  \int_{\partial \M}  f^2(\textbf{y})  \bar{R}_{\delta}(\textbf{x},\textbf{y})  d \tau_\textbf{y}  d \mu_\textbf{x}  \ \Big]^{\frac{1}{2}}  \\
& \leq \delta^{\frac{1}{2}} \Big[ \int_{\M} w_{\delta}^2(\textbf{x}) d \mu_\textbf{x} \ \int_{\partial \M}    f^2(\textbf{y})  d \tau_\textbf{y}   \ \Big]^{\frac{1}{2}} 
\leq \delta^{\frac{1}{2}} \left \lVert f \right \rVert_{L^{2}( \partial \M)} \left \lVert w_{\delta} \right \rVert_{L^2(\M)} \\ &
\leq \left \lVert f \right \rVert_{H^1( \M)} \left \lVert w_{\delta} \right \rVert_{L^2(\M)};
\end{split}
\end{equation}

in addition, we have
\begin{equation} \label{sm3}
\begin{split}
& \int_{\M} \  w_{\delta}(\textbf{x}) \ \mathcal{G}_{\delta} (\frac{ \mathcal{Q}_{\delta} f(\textbf{x}) }{\tilde{R}_{\delta}(\textbf{x})} ) \  d \mu_\textbf{x}   \\ & =\int_{\M} w_{\delta}(\textbf{x}) \ \int_{\partial \M} \frac{ \mathcal{Q}_{\delta} f(\textbf{y}) }{\tilde{R}_{\delta}(\textbf{y})}   (2+\kappa_{\n} (\textbf{y})(\textbf{x}-\textbf{y}) \cdot \n(\textbf{y})) \ \bar{R}_{\delta}(\textbf{x},\textbf{y}) d \tau_\textbf{y} \ d \mu_\textbf{x} \\
& \leq  \int_{\M} \ \vert w_{\delta}(\textbf{x})  \vert \ \int_{\partial \M}  \frac{C \ \delta^2 }{\delta} \ \vert f(\textbf{y}) \vert \ 3 \ \bar{R}_{\delta}(\textbf{x},\textbf{y}) d \tau_\textbf{y} \ d \mu_\textbf{x} \\ &
\leq C \delta \int_{\partial \M} \ \vert f(\textbf{y}) \vert  \ \int_{\M} \ \vert w_{\delta}(\textbf{x})  \vert \ \bar{R}_{\delta}(\textbf{x},\textbf{y})  d \mu_\textbf{x} d \tau_\textbf{y} \\
& \leq C \delta \Big[ \int_{\partial \M}  f^2(\textbf{y}) d \tau_\textbf{y} \int_{\partial \M} ( \int_{\M} \ w_{\delta}^2(\textbf{x}) \ \bar{R}_{\delta}(\textbf{x},\textbf{y})  d \mu_\textbf{x}  ) \ ( \int_{\M}  \ \bar{R}_{\delta}(\textbf{x},\textbf{y})  d \mu_\textbf{x} )    d \tau_\textbf{y} \Big]^{\frac{1}{2}} \\
& \leq C \delta  \Big[ \int_{\partial \M}  f^2(\textbf{y}) d \tau_\textbf{y}  \int_{\M} \ \frac{1}{\delta} \ w_{\delta}^2(\textbf{x}) d \mu_\textbf{x}    \Big]^{\frac{1}{2}} \leq C \delta^{\frac{1}{2}} \left \lVert  f \right \rVert_{L^2(\partial \M)} \left \lVert w_{\delta} \right \rVert_{L^2(\M)}  \\ &
\leq C \left \lVert f \right \rVert_{H^1(\M)} \left \lVert w_{\delta} \right \rVert_{L^2(\M)},
\end{split}
\end{equation}

The above three inequalities reveal
\begin{equation} \label{c31}
\int_{\M} w_{\delta}(\textbf{x}) \mathcal{P}_{\delta} f(\textbf{x}) d \mu_\textbf{x} + \int_{\M} w_{\delta}(\textbf{x}) \ \mathcal{G}_{\delta}  (\frac{ \mathcal{Q}_{\delta} f(\textbf{x}) }{\tilde{R}_{\delta}(\textbf{x})} ) d \mu_\textbf{x} \leq C \left \lVert f \right \rVert_{H^1(\M)} \left \lVert w_{\delta} \right \rVert_{L^2(\M)}.
\end{equation}
\end{proof}

\end{document}